\documentclass[reqno,12pt]{amsart}
\usepackage{a4wide}
\usepackage{amsmath}
\usepackage{amssymb}
\usepackage{amsthm}

\numberwithin{equation}{section}

\newtheorem{theo}{Theorem}
\newtheorem{conj}{Conjecture}
\newtheorem{coro}{Corollary}
\newtheorem{prop}{Proposition}

\newtheorem{lem}{Lemma}

\theoremstyle{remark}
\newtheorem{remark}{Remark}
\newtheorem*{Remark}{Remark}
\newtheorem*{Remarks}{Remarks}

\def\al{\alpha}
\def\be{\beta}
\def\ga{\gamma}

\def\ep{\varepsilon}

\def\la{\lambda}

\def\ph{\varphi}

\def\Ga{\Gamma}
\def\De{\Delta}
\def\Th{\Theta}

\def\({\left(}
\def\){\right)}
\def\[{\left[}
\def\]{\right]}

\def\fl#1{\left\lfloor#1\right\rfloor}

\def\lcm{\operatorname{lcm}}

\def\dd{\textup{d}}

\begin{document}

\title[]{On the integrality of the Taylor coefficients of mirror maps}

\author[]{C. Krattenthaler$^\dagger$ and T. Rivoal}
\date{\today}

\address{C. Krattenthaler, Fakult\"at f\"ur Mathematik, Universit\"at Wien,
Nordbergstra{\ss}e~15, A-1090 Vienna, Austria.
WWW: \tt http://www.mat.univie.ac.at/\~{}kratt.}

\address{T. Rivoal,
Institut Fourier,
CNRS UMR 5582, Universit{\'e} Grenoble 1,
100 rue des Maths, BP~74,
38402 Saint-Martin d'H{\`e}res cedex,
France.\newline
WWW: \tt http://www-fourier.ujf-grenoble.fr/\~{}rivoal.}

\thanks{$^\dagger$Research partially supported 
by the Austrian Science Foundation FWF, grants Z130-N13 and S9607-N13,
the latter in the framework of the National Research Network
``Analytic Combinatorics and Probabilistic Number Theory''}



\begin{abstract}
We show that the Taylor coefficients of the series
${\bf q}(z)=z\exp({\bf G}(z)/{\bf F}(z))$ are integers, where ${\bf F}(z)$ and 
${\bf G}(z)+\log(z) {\bf F}(z)$ 
are specific solutions of certain hypergeometric differential
equations with maximal unipotent monodromy at $z=0$. 
We also address the question of finding the largest integer $u$
such that the Taylor coefficients of $(z ^{-1}{\bf q}(z))^{1/u}$ are still
integers.
As consequences, we are able to prove numerous integrality
results for the Taylor coefficients of mirror maps 
of Calabi--Yau complete intersections in weighted projective spaces, 
which improve and refine previous results by Lian and Yau, and by Zudilin.
In particular, we prove the general ``integrality'' conjecture of 
Zudilin about these mirror maps.
\end{abstract}

\maketitle

\section{Introduction and statement of results}

\subsection{Mirror maps}
{\em Mirror maps} have appeared quite recently in mathematics and
phys\-ics. Indeed, 
the term ``mirror map'' was coined in the late  
1980s
by physicists whose research in string theory led them to discover
deep facts in algebraic geometry 
(e.g., given a Calabi--Yau threefold $M$, 
they constructed another Calabi--Yau threefold, the ``mirror'' of $M$, 
whose properties can be used to enumerate the rational curves on $M$).

The purpose of the present article is to prove rather sharp
integrality assertions for the Taylor coefficients of 
mirror maps coming from certain
hypergeometric differential equations, which are 
Picard--Fuchs equations of suitable one parameter families 
of Calabi--Yau complete intersections in weighted projective spaces. 
The corresponding results (see 
Theorems~\ref{thm:4} and \ref{thm:2}) encompass 
integrality results on these mirror maps which exist in the
literature, improving and refining them in numerous instances.

In a sense, mirror maps can be viewed as higher 
order generalisations of 
certain classical modular forms (defined over various 
congruence sub-groups of $SL_2(\mathbb{Z})$), the latter appearing naturally 
at low order in Schwarz's theory of 
hypergeometric functions (see~\cite{yoshida}). 
For integers $k\ge 1$ and $N\ge 1$, let us define the power 
series
\begin{equation*}
F(z):=\sum_{m=0}^{\infty} \frac{(Nm)!^k}{m!^{kN}} \,z^m,
\end{equation*}
which converges for $\vert z\vert <1/N^{kN}.$ 
The function $F(z)$ is solution of a hypergeometric differential equation
of degree $kN$, 
which is a special case of~\eqref{eq:equadiff} below. 
The equation has maximal unipotent monodromy 
(MUM) 
(in particular, the roots of the indicial equation at $z=0$ are all $0$).
A basis of solutions with at most logarithmic
singularities around $z=0$ 
can then be obtained by Frobenius' method; see~\cite{yoshida}. In
particular, there exists another solution of the form
$G(z)+\log(z)F(z)$, 
where $G(z)$ is holomorphic around~$0$,
\begin{equation*}
G(z):=\sum_{m=1}^{\infty} \frac{(Nm)!^k}{m!^{kN}}
kN(H_{Nm}-H_m)\,z^m,
\end{equation*}
with 
$H_n:=\sum _{i=1} ^{n}\frac {1} {i}$ denoting the $n$-th harmonic number.
In the context of mirror symmetry, the function $q(z):=z\exp(G(z)/F(z))$
is usually called  
{\em canonical coordinate}, and its compositional
inverse $z(q)$ is  
the prototype of a {\it mirror map}.

In the case $k=N=2$, one can express $z(q)$ explicitly
in terms of the Legendre function
$\lambda(q):=16q\prod_{n=1}^{\infty}\big((1+q^{2n})/(1+q^{2n-1})\big)$,
namely as $z(q)= \lambda(q)/16$, which is modular over $\Gamma(2)$. 
Moreover, if $k=3$ and $N=2$, we have
$z(q)=\lambda(q)(1-\lambda(q))/16$. 
In particular, in both cases the power series $q(z)$ and $z(q)$
have {\em integral Taylor coefficients}. 
It is  
this fact that we will generalise in this paper.
For other examples of mirror maps of modular origin, we refer to the  
discussion in~\cite[pp.~111--113]{andre} (and also for the importance
of such facts in Diophantine approximation) and \cite[Sec.~3]{lianyau}.

The most famous (apparently non-modular) example of
a mirror map arises in the case when $N=5$ and $k=1$, which was used  
in the epoch-making paper 
by the physicists Candelas et al.~\cite{candelas}  in their study 
of the family ${\mathbf M}$ of quintic hypersurfaces 
in $\mathbb{P}^4(\mathbb{C})$ defined by $\sum_{k=1}^5x_k^5-5z
\prod_{k=1}^5x_k=0$, $z$ being a complex parameter
(see~\cite{morrisonjams, pandha, VoisAA}). 

The following conjecture belongs probably to the folklore of mirror
symmetry theory.
\begin{conj}
\label{conj:zudilin1}
For any integers $k\ge 1$ and $N\ge 1$, we have 
$q(z)\in z\mathbb{Z}[[z]]$ and $z(q)\in q\mathbb{Z}[[q]]$.~{\em(}\footnote{%
In the number-theoretic study undertaken in the present paper,
we are interested in the integrality of the coefficients of (roots of) mirror
maps $z(q)$. In that context, the mirror map $z(q)$ and the
corresponding canonical coordinate $q(z)$
play strictly the same role, because $(z^{-1}q(z))^{1/\tau}\in 
1+z\mathbb{Z}[[z]]$ for some integer  
$\tau$ implies that $(q^{-1}z(q))^{1/\tau}\in 1+q\mathbb{Z}[[q]]$, and
conversely. (See~\cite[Introduction]{lianyau1}.)
We shall, in the sequel, formulate our integrality results exclusively for
canonical coordinates, assuming tacitly that the reader keeps in mind
that they automatically also
hold for the corresponding mirror maps.  
It is also therefore that, by abuse of terminology, we 
shall often use the term ``mirror map'' for any  
canonical coordinate.}{\em)}  

\end{conj}

Lian and Yau \cite[Sec.~5, Theorem~5.5]{lianyau} proved this
conjecture for $k=1$ and any $N$ which is a prime number.
Zudilin~\cite[Theorem~3]{zud} extended their result by proving the
conjecture for any $k\ge 1$
and any $N$ which is a prime power.

Our original goal was to settle 
Conjecture~\ref{conj:zudilin1} in complete generality.
In the present paper, we shall accomplish much more.
We shall prove a more general conjecture by Zudilin \cite{zud} 
concerning the integrality of Taylor coefficients of a very large
class of mirror maps (see
Conjecture~\ref{conj:zudilin2}) and refinements of the corresponding
integrality results in special cases.
In the remainder of this introductory section, we describe these two
sets of results (see Theorems~\ref{thm:4} and \ref{thm:2}). 
Their proofs will then be
given in the subsequent sections.

\subsection{Zudilin's conjecture} \label{ssec:zudconj}

In order to state Zudilin's conjecture,
we need to introduce some notation.

For a positive integer $N$,
let $p_1, p_2, \dots, p_\ell$ denote its distinct prime factors.   
We define the vectors of integers
\begin{equation}
(\alpha_j)_{j=1, \ldots, \mu}  := \left( N, \frac{N}{p_{j_1}p_{j_2}},
\frac{N}{p_{j_1}p_{j_2}p_{j_3}p_{j_4}}, \ldots 
\right)_{1\le j_1<j_2<\dots \le \ell } \label{eq:aj}
\end{equation}
and
\begin{equation}
(\beta_j)_{j=1, \ldots, \eta} := \left(
\frac{N}{p_{j_1}}, \frac{N}{p_{j_1}p_{j_2}p_{j_3}}, \ldots,  1,1, \dots, 1 
\right)_{1\le j_1<j_2<\dots \le \ell }, \label{eq:bj}
\end{equation}
where $\al_1+\al_2+\dots + \al_\mu=\be_1+\be_2+\dots +\be_\eta.$ 
We put $\mathbf{B}_1(m):=1$ and $\mathbf{H}_1(m):=0$, and, if $N\ge2$,
\begin{equation}\label{eq:Bzudilin}
\mathbf{B}_N(m) := \frac{\prod_{j=1}^\mu (\alpha_j
m)!}{\prod_{j=1}^\eta (\beta_j m)!} 
\end{equation}
and~(\footnote{Zudilin's definition \cite[Eq.~(5)]{zud} of the quantity 
$\mathbf H_N(m)$ (he writes $D_N(m)$) is
different from \eqref{eq:rajout2}. We do not need it in our paper,
but, for the sake of completeness, we
prove the equivalence of the two definitions in Section~\ref{sec:H_N}.})
\begin{equation}\label{eq:rajout2}
\mathbf{H}_N(m)
= \sum_{j=1}^{\mu} \alpha_j H_{\alpha_jm} -  \sum_{j=1}^{\eta} \beta_j
H_{\beta_jm}.
\end{equation}
For example, we have 
$$
\mathbf{B}_4(m) = \frac{(4m)!}{(2m)!\,m!^2}, \; \mathbf{B}_6(m) =
\frac{(6m)!}{(3m)!\,(2m)!\,m!},  
\; \mathbf{B}_{30}(m) =
\frac{(30m)!\,(5m)!\,(3m)!\,(2m)!}{(15m)!\,(10m)!\,(6m)!\,m!^9},
$$
and, correspondingly,
\begin{multline*}
\mathbf{H}_4(m) = 4H_{4m}-2H_{2m}-2H_m, \; \mathbf{H}_6(m) =
6H_{6m}-3H_{3m}-2H_{2m}-H_m,  
\\ 
\mathbf{H}_{30}(m) =
30H_{30m}+5H_{5m}+3H_{3m}+2H_{2m}-15H_{15m}-10H_{10m}-6H_{6m}-9H_m.
\end{multline*}
Given a vector $\mathbf N=(N_1,N_2,\dots,N_k)$ of positive integers,
we construct the power series
\begin{equation*} 
\mathbf{F}_{\mathbf{N}}(z) := \sum_{m=0}^{\infty} 
\bigg(
\prod_{j=1}^k
\mathbf{B}_{N_j}(m)
\bigg)
z^m  
\end{equation*}
and
$$
\mathbf{G}_{\mathbf{N}}(z) := \sum_{m=1}^{\infty} \bigg(\sum_{j=1}^k
\mathbf{H}_{N_j}(m) \bigg)\bigg( \prod_{j=1}^k \mathbf{B}_{N_j}(m)\bigg)
z^m.  
$$
It can be shown (see \cite{zud}, taking into account
Lemma~\ref{lem:rajout2} in Section~\ref{sec:H_N}) that
the series $\mathbf{F}_{\mathbf{N}}(z)$ and 
$\mathbf{G}_{\mathbf{N}}(z)+
\log (z)\,\mathbf{F}_{\mathbf{N}}(z)$ are two 
solutions to the equation $\mathbf Ly=0$,
where the hypergeometric differential operator ${\bf L}$ is defined by 
\begin{equation}\label{eq:equadiff}
{\bf L}:=\left(z\frac{\dd}{\dd z}\right)^{\varphi(N_1)+\cdots +\varphi(N_k)} 
- C_{\bf N} z
\prod_{j=1}^k\prod_{i=1}^{\varphi(N_j)} 
\left(z\frac{\dd}{\dd z} + \frac{r_{i,j}}{N_j}\right),
\end{equation}
where $\varphi(\,.\,)$ is Euler's totient function, 
$C_{\bf N}:=\prod _{j=1} ^{k} C_{N_j}$ with $C_{N_j}:={N_j}^{\varphi(N_j)}
{\prod_{p\mid N_j}} p^{\varphi(N_j)/(p-1)}$, and 
the $r_{i,j}\in \{1, 2, \ldots, N_j\}$ form the residue 
classes modulo $N_j$ which are coprime to $N_j$. 
The differential equation ${\bf L}y=0$ has 
MUM~(\footnote{%
This equation is the Picard--Fuchs equation 
of the mirror Calabi--Yau family of 
a one parameter family of Calabi--Yau complete intersections in a 
weighted projective space. See~\cite{corti}
and~\cite[Sec.~3]{hosono}.}) at the origin. 

We can now state Zudilin's conjecture 
from \cite[p.~605]{zud}.
(Zudilin's formulation is different. That our formulation is equivalent
with Zudilin's 
follows from \cite[Lemma~4]{zud}, respectively \eqref{eq:rajout3},
and from Lemma~\ref{lem:rajout2} in Section~\ref{sec:H_N}.)

\begin{conj}[\sc Zudilin]\label{conj:zudilin2}
For any positive integers $N_1, N_2, \dots, N_k$, we have 
$
{\bf q}_{\mathbf{N}} (z):= z 
\exp(\mathbf{G}_{\mathbf{N}}(z)/\mathbf{F}_{\mathbf{N}}(z)) \in
z\mathbb{Z}[[z]]. 
$
\end{conj}

It can be seen (see \cite[paragraph above Theorem~2]{zud},
respectively \eqref{eq:FN}--\eqref{eq:GLN} below) that
Conjecture~\ref{conj:zudilin1} is the special case of the above
conjecture where $k$ is replaced by $k\cdot d(N)$ (with $d(N)$
denoting the number of positive divisors of $N$), and where 
$\{N_1,\dots,N_{k\cdot d(N)}\}$ is the multiset (\footnote{A multiset 
is a ``set'' where one allows repetitions of elements.}) in which each
divisor of $N$ appears exactly $k$ times. 

Zudilin proved that his conjecture holds 
under the condition that if a prime number divides $N_1N_2\cdots
N_k$ then it also divides each $N_j$. 

We claim that Conjecture~\ref{conj:zudilin2} follows from the theorem
below, which is the first main result of the paper. For the statement
of the theorem,  
for an integer $L\ge 1$, we need to define  
$$
\mathbf{G}_{L, \mathbf{N}}(z) := \sum_{m=1}^{\infty} H_{Lm}
\bigg(\prod_{j=1}^k \mathbf{B}_{N_j}(m)\bigg) z^m. 
$$

\begin{theo}\label{thm:4} For any integers $N_1,N_2, \ldots, N_k\ge 1$ and 
$L\in \{1, 2, \ldots, \max(N_1, \ldots, N_k)\}$, we have 
$
{\bf q}_{L,\mathbf{N}} (z) :=
\exp(\mathbf{G}_{L,\mathbf{N}}(z)/\mathbf{F}_{\mathbf{N}}(z) 
)
\in\mathbb{Z}[[z]].
$
\end{theo}

An outline of the proof of this theorem is given in 
Section~\ref{sec:7}, with details being filled 
in in Sections~\ref{sec:1}--\ref{sec:6a}.

To see that Theorem~\ref{thm:4} implies Conjecture~\ref{conj:zudilin2}, 
we observe that, by \eqref{eq:rajout2},
$\sum_{j=1}^k \mathbf{H}_{N_j}(m)$ is a finite sum of terms of the form 
$\lambda H_{Lm}$, where $\la$ and $L$ are integers with 
$$
L\in \{1, 2,\dots, \max(N_1, \ldots, N_k)\}.
$$
Thus, $z^{-1}{\bf q}_{\mathbf{N}} (z)$ is a product of series 
$\mathbf{q}_{L, \mathbf{N}}(z)$, each one raised to an integer power. It follows that 
Conjecture~\ref{conj:zudilin2} implies Theorem~\ref{thm:4}, as claimed.

\subsection{Stronger integrality properties in special cases}
\label{sec:1.3}

Let $N_1,N_2,\dots,N_k$ be given positive integers (not necessarily
distinct). 
In the setting of Section~\ref{ssec:zudconj}
with $k$ replaced by $k\cdot (d(N_1)+d(N_2)+\dots+d(N_k))$, 
$d(N)$ again denoting the number of positive divisors of $N$, 
we consider the special case
where the vector $\mathbf{N}$ can be partitioned into $k$ 
blocks, the $j$-th block consisting of all the positive 
divisors of $N_j$, $j=1,2,\dots,k$. 
It can be seen
(cf.\ 
\cite[paragraph above Theorem~2]{zud}) that 
the functions $\mathbf{F}_{\mathbf{N}}(z)$, $\mathbf{G}_{\mathbf{N}}(z)$ and 
$\mathbf{G}_{L,\mathbf{N}}(z)$ then simplify to
\begin{gather}
 \sum_{m=0}^{\infty} 
\bigg(
\prod_{j=1}^k
\frac{(N_jm)!}{m!^{N_j}}  
\bigg) 
z^m, 
\label{eq:FN}
\\
\sum_{m=1}^{\infty} 
\bigg(\sum_{j=1}^kN_j(H_{N_jm}-H_m)\bigg)\bigg(\prod_{j=1}^k
\frac{(N_jm)!}{m!^{N_j}}  \bigg) z^m,
\label{eq:GN}
\\
\sum_{m=1}^{\infty} H_{Lm} \bigg(\prod_{j=1}^k
\frac{(N_jm)!}{m!^{N_j}}  \bigg) z^m,
\label{eq:GLN}
\end{gather}
respectively. In order to simplify notation, for the remainder of this
subsection we 
``redefine'' $\mathbf{N}$ by letting 
$\mathbf{N}=(N_1,N_2, \ldots, N_k)$, and we denote the
series in \eqref{eq:FN}, \eqref{eq:GN}, and \eqref{eq:GLN} respectively
by $F_{\mathbf{N}}(z)$, $G_{\mathbf{N}}(z)$, and $G_{L,\mathbf{N}}(z)$. 
Accordingly, we define 
$q_{\mathbf{N}}(z):=z\exp\big(G_{\mathbf{N}}(z)/F_{\mathbf{N}}(z)\big)$
and
$q_{L,\mathbf{N}}(z):=\exp\big(G_{L,\mathbf{N}}(z)/F_{\mathbf{N}}(z)\big).$ 

For the mirror map $q_{(N)}(z)$ arising for $k=1$,
physicists made the observation that, apparently,
\begin{equation}\label{eq:raflianyau}
\big(z^{-1}q_{(N)}(z)\big)^{1/N}\in \mathbb{Z}[[z]].
\end{equation}
This was proved by Lian and Yau \cite{lianyau2} for any prime number $N$, 
thus strengthening their result from~\cite{lianyau} mentioned after
Conjecture~\ref{conj:zudilin1}.
The observation \eqref{eq:raflianyau} leads naturally to the more
general question of determining the largest integer $V$ 
such that $\left(z ^{-1}q(z) \right) ^{1/V}\in 
\mathbb{Z}[[z]]$ for the mirror map $q(z)=q_{(N)}(z)$ in
\eqref{eq:raflianyau} or other mirror 
maps.~(\footnote{Let $q(z)$ be a given power 
series in $\mathbb{Z}[[z]]$, and let
$V$ be the largest integer with the property that $q(z)^{1/V}\in \mathbb
Z[[z]]$. Then $V$ carries complete information about {\it all\/}
integers with that property: namely,
the set of integers $U$ with $q(z)^{1/U}\in \mathbb
Z[[z]]$ consists of all divisors of $V$. Indeed, it is clear that all
divisors of $V$ belong to this set. Moreover, 
if $U_1$ and $U_2$ belong to this set, then also $\lcm(U_1,U_2)$ does
(cf.\ \cite[Lemma~5]{HeRSAA} for a simple proof 
based on B\'ezout's lemma).})    
While we are not able to give a precise answer, 
we shall prove the following result for a large class of mirror maps
which, as we explain in \cite{kratriv3}, comes very close 
to being optimal for this class.

\begin{theo} \label{thm:2}
For any integers $N_1, \ldots, N_k\ge 1$,
let $M_{\mathbf{N}}=\prod_{i=1}^k N_i!$. Furthermore,
let $\Th_L:=L!/\gcd(L!, L!\,H_L)$ be the denominator of $H_L$
when written as a reduced fraction.
Then, for all
$L\in \{1, 2, \ldots, \max(N_1, \ldots, N_k)\}$, we have 
$q_{L,\mathbf{N}}(z)^{\frac{\Th_L}{M_{\mathbf{N}}}} \in\mathbb{Z}[[z]].$
\end{theo}

\begin{Remarks} 
(a) For any integer $s \ge 1$, 
we have $q_{L,\mathbf{N}}(z)^{1/s}=1+s^{-1}M_{{\bf N}}H_L z + \mathcal{O}(z^2)$, 
and hence  Theorem~\ref{thm:2} is optimal when $L=1$. This is not necessarily the case for 
other values of $L$ and, in particular, Theorem~\ref{thm:2} can be improved 
when $L=N_1=\cdots =N_k$. 
See~\cite{kratriv3} for such results.

\smallskip
(b)
It is natural to expect 
refinements of Theorem~\ref{thm:4} in the spirit of
Theorem~\ref{thm:2}. We did not make a systematic research in
this direction, but it could be  
interesting to do so. For example, in the case $k=1$ and $\mathbf{N}=(6)$, 
it seems that the following relations are best possible:
${\bf q}_{1,(6)}(z)^{1/60}$,  
${\bf q}_{2,(6)}(z)^{1/6}$,  
${\bf q}_{3,(6)}(z)^{1/2}$, 
${\bf q}_{4,(6)}(z)$, 
${\bf q}_{5,(6)}(z)$ and 
${\bf q}_{6,(6)}(z)$ are in $\mathbb{Z}[[z]]$. 
As a first step towards such refinements, 
we prove in Lemma~\ref{lem:diviBB} in Section~\ref{sec:8} that
$\mathbf{B}_{\mathbf{N}}(1)$ always divides  
$\mathbf{B}_{\mathbf{N}}(m)$ for any $m\ge 1$ and any $\mathbf{N}$,
where $\mathbf{B}_{\mathbf{N}}(m):=\prod_{j=1}^k \mathbf B_{N_j}(m)$.
Our techniques enable us to deduce that  
${\bf q}_{1,\mathbf{N}}(z)^{1/\mathbf{B}_{\mathbf{N}}(1)} \in
\mathbb{Z}[[z]]$
(cf.\ Remark~\ref{foot:1} in Section~\ref{sec:7}), 
which proves the above assertion that 
${\bf q}_{1,(6)}(z)^{1/60}\in \mathbb{Z}[[z]]$.
This is optimal because ${\bf q}_{1,\mathbf{N}}(z)=
1+\mathbf{B}_{\mathbf{N}}(1)z+\mathcal{O}(z^2)$. 
It turns out that 
$\mathbf{B}_{\mathbf{N}}(1)$ is a natural generalisation of
the quantity $M_{\mathbf{N}}$ which
appears in Theorem~\ref{thm:2}. However, for larger values of the
parameter $L$, we do not know what 
the analogue of the quantity $M_{\mathbf{N}}/\Th_L$
appearing in Theorem~\ref{thm:2} would be.
\end{Remarks}

An outline of the proof of Theorem~\ref{thm:2} is given in
Section~\ref{sec:2}, with details being filled in in
Sections~\ref{sec:4}--\ref{sec:6}.

Due to the equation
\begin{equation}\label{eq:truemap}
q_{(N,\dots,N)}(z)=zq_{N,(N,\dots,N)}(z)^{kN}q_{1,(N,\dots,N)}(z)^{-kN}, 
\end{equation} 
(with $k$ occurrences of $N$ in $(N,\dots,N)$),
Theorem~\ref{thm:2} has the following
consequence for the mirror map $q_{(N,\dots,N)}(z)$, 
thus improving significantly upon \eqref{eq:raflianyau}.

\begin{coro} \label{coro:1}
For all integers $k\ge 1$ and $N\ge 1$, we have 
$$
\big(z^{-1}q_{(N,\dots,N)}(z)\big)^{\frac{\Th_N}{N!^k kN}} \in \mathbb{Z}[[z]].
$$
\end{coro}

\begin{Remarks}
(a) In particular, in the emblematic case of the mirror map $q_{(5)}(z)$ of
the quintic (case $N=5$, $k=1$), we obtain that  
$\big(z^{-1}q_{(5)}(z)\big)^{1/10} \in \mathbb{Z}[[z]]$, which improves
on~\eqref{eq:raflianyau} by a factor of $2$. 
\smallskip

(b)
Also Corollary~\ref{coro:1} can be improved. 
The corresponding result, which is optimal subject to a widely believed
conjecture on harmonic numbers, 
is very technical. We refer the reader again to our
article~\cite{kratriv3}.  
\end{Remarks}

\subsection{Method of proof}
\label{sec:method}

The basic idea is to transfer the original integrality assertions 
into a $p$-adic framework 
(by means of Lemma~\ref{lem:1}), to a point where we can 
employ Dwork's theory of formal congruences. This theory
seems to provide the most powerful tools available for attacking integrality
assertions of the type discussed in this paper.
We recall the corner stones of Dwork's theory in Section~\ref{sec:1}.

We draw the reader's attention to the fact that, while the general
line of our approach follows that of previous authors
(particularly \cite{lianyau2}), there
does arise a crucial difference (other than just technical
complications): the reduction and rearrangement of the sums
$\mathbf C(a+Kp)$ in Section~\ref{sec:7},
respectively $C(a+Kp)$ in Section~\ref{sec:2}, 
via the congruence \eqref{eq:J} 
require a new reduction step, namely
Lemma~\ref{lem:12a}, respectively Lemma~\ref{lem:12}. In fact, the
proofs of these two lemmas 
form the most difficult parts of our
paper. (In previous work, the
use of \eqref{eq:J} sufficed because the corresponding authors
restricted themselves to $N$ being prime or a prime power.)

\subsection{Related work and perspectives}
\label{sec:relwork}

One of the goals of the present paper is to highlight
number-theoretic phenomena in the context of mirror symmetry theory.
Our analysis is clearly on the number theory side.
Nevertheless, we hope that our results contribute to the clarification
of such phenomena. On the other hand, it should not be hidden that 
integrality phenomena should also have intrinsic
significance, within mirror symmetry. One such example is the
(apparent) coincidence of coefficients in Yukawa couplings with
numbers of certain rational curves (cf.\ 
\cite{givental,lianliuyau}; see also the discussion in~\cite[p.~49]{VoisAA}). 

In the context of our results, it is possible to prove similar,
but strictly weaker, statements by means of methods of arithmetic 
geometry. This is the case for the  
preprint~\cite{volog}, which is an elaborate version of
\cite{KoSVAA}. The mirror maps that are considered in that paper comprise 
ours. When both approaches apply simultaneously, our results in 
Theorems~\ref{thm:4} and \ref{thm:2} 
are stronger than Theorem~2 in \cite[Sec.~1.3]{volog}.
Indeed, we prove that certain mirror maps have integral Taylor
coefficients, while in \cite{volog} the weaker statement is proved that 
mirror maps have Taylor coefficients in $\mathbb{Z}[1/n]$, where
$n$ is an integer parameter of geometric origin which is at  
least $2$ (by assumption iii) just before Theorem~2 in \cite{volog}).

We also want to point out that the range of application of
Dwork's ideas is not restricted to 
$p$-adic functions in one variable. In \cite{KrRiAG}, we extend
Dwork's theory as outlined in Section~\ref{sec:1} to several variables. As 
applications, we obtain integrality properties of mirror maps in several 
variables arising in the context of the very general 
multivariable mirror maps coming from 
the Gelfand--Kapranov--Zelevinsky hypergeometric series 
(see~\cite[Sec.~7.1]{batstrat},~\cite{hosono} and~\cite[Sec.~8]{stienstra} 
for numerous examples related to Calabi--Yau manifolds which are 
complete intersections in products of weighted projective spaces). As
a by-product, by appropriate 
specialisations, we are even able to prove 
(predicted) integrality of the Taylor coefficients of some mirror maps in one 
variable that do not fall under the scope of the results of the present paper, 
in particular many of those in the table presented in~\cite{aesz}.

We therefore believe that Dwork's methods provide valuable insight in
integrality properties of mirror maps. Since the power and range of
applicability of these methods have apparently not yet been exhausted,
their development should be further pursued. In particular, we hope to be 
able to test them against the difficult number-theoretical 
problems raised by Yukawa couplings.

Finally, we point out that the quantity $\mathbf B_N(m)$
defined in \eqref{eq:Bzudilin}, which is the crucial building
block for the series $\mathbf F_{\mathbf N}(z)$,
$\mathbf G_{\mathbf N}(z)$, and $\mathbf G_{L,\mathbf N}(z)$, is in
fact an integer for all $N$ and $m$. This follows from the criterion
\cite{LandAA} of Landau, which, applied to our case, says that, for a fixed $N$, 
$\mathbf B_N(m)$ is integral for all $m$ if and only if
$$\sum_{j=1}^\mu \fl{\alpha_j x}-\sum_{j=1}^\eta \fl{\beta_j x}\ge0
$$
for all non-negative real numbers $x$. Indeed, the above quantity is
exactly the quantity $\Delta(x)$ defined in Lemma~\ref{lem:10a} in
Section~\ref{sec:8}, which plays a crucial role in the proof of
Lemma~\ref{lem:12a} in Section~\ref{sec:9}, and, thus, in the proof
of Theorem~\ref{thm:4}. More precisely, amongst its properties, we 
use many times 
the fact that it is weakly increasing on $[0,1)$. 
We point out that such ``Landau functions'' were introduced 
in the theory of mirror symmetry of Calabi--Yau threefolds 
by Rodriguez-Villegas~\cite{villegas}, who used them to 
prove that there are exactly 14 hypergeometric functions whose 
coefficients can be written as integral quotients of factorials (after 
rescaling the variable $z$ to $Cz$ for some $C$) 
and with the MUM property at the origin. 
In~\cite{villegas2}, he also used Landau functions to 
characterise algebraic hypergeometric series  whose 
coefficients can be written as integral quotients of factorials after 
rescaling:
using the theory of Beukers and Heckman~\cite{beukheck}, he proved that 
such series are algebraic if and only if 
$\Delta(x)\in\{0,1\}$. Of course, it is not always possible to write 
the Taylor coefficients of a hypergeometric series as quotients of 
factorials.

It can be proved that amongst hypergeometric series 
whose Taylor coefficients are integral quotients of factorials, 
the weak 
increase of the associated Landau function $\Delta$ is 
equivalent to the MUM property. This is a consequence of more general 
results of Eric Delaygue, in a thesis currently 
being written under the supervision of the second author.
Thus, from the point of view of mirror symmetry where the MUM property is 
essential, our Theorem~\ref{thm:4} is 
best possible in that class of hypergeometric series.
The question then naturally arises which properties the
Landau function $\Delta$ must have such that 
the mirror map formally associated to it (by forming 
functions 
analogous to $\bf F_{\bf N}$ and $\bf G_{\bf N}$) has integral Taylor 
coefficients 
(even when a mirror symmetry interpretation is not available;
for example, when the MUM property does not hold).
We believe that the following statement  
is true: the formal mirror map has integral Taylor coefficients if and only if 
$\Delta(x)$, defined on $[0,1)$, remains $\ge 1$ 
after its first jump from $0$.
New ideas are necessary to solve this problem because the approach used 
in this paper does not work in this more general setting; indeed there are counterexamples 
to almost all our lemmas when $\Delta$ is not weakly increasing. 
This is currently under investigation by Eric Delaygue.

\subsection{Structure of the paper}

We now briefly review the organisation of 
the rest of the paper.

The proofs of our theorems being highly complex,
we start with brief outlines of the proofs of Theorems~\ref{thm:4} and
\ref{thm:2} in Sections~\ref{sec:7} and \ref{sec:2}, respectively.
These outlines reduce the proofs to a certain number of lemmas. 
The reduction is
heavily based on Dwork's theory of formal congruences, the relevant 
pieces of which being recalled in Section~\ref{sec:1}. 
The lemmas which are necessary for the proof of Theorem~\ref{thm:4} 
are subsequently established in 
Sections~\ref{sec:8}--\ref{sec:6a},
while those necessary for the proof of Theorem~\ref{thm:2} are established in
Sections~\ref{sec:4}--\ref{sec:6}.

\section{Outline of the proof of Theorem~\ref{thm:4}}
\label{sec:7}

In this section, we provide a brief outline of the proof of
Theorem~\ref{thm:4}. As we already said in the introduction, the
proof follows the $p$-adic approach pioneered by Dwork
\cite{dworkihes,dwork}, of which we review its corner stones in 
Section~\ref{sec:1}. We show that this approach allows us to reduce the
proof of Theorem~\ref{thm:4} to
Lemmas~\ref{lem:12a}--\ref{lem:strat4}. These lemmas are subsequently
proved in Sections~\ref{sec:9}--\ref{sec:6a}, with four auxiliary lemmas
being the subject of Section~\ref{sec:8}.

By Dwork's Lemma 
given in Section~\ref{sec:1} (or rather its consequence given in
Lemma~\ref{lem:4}), we want to prove that  
$$\mathbf F_{\mathbf N}(z)\mathbf G_{L,\mathbf N}(z^p)
-p\mathbf F_{\mathbf N}(z^p)\mathbf G_{L,\mathbf
N}(z) \in p z \mathbb{Z}_p[[z]].$$ 

The $(a+Kp)$-th Taylor coefficient of
$\mathbf F_{\mathbf N}(z)\mathbf G_{L,\mathbf N}(z^p)
-p\mathbf F_{\mathbf N}(z^p)\mathbf G_{L,\mathbf
N}(z)$ is  
\begin{equation} \label{eq:firstreductiona} 
\mathbf C(a+Kp):=\sum_{j=0}^K 
\mathbf B_{\mathbf N}(a+jp)\mathbf B_{\mathbf N}(K-j) 
(H_{L(K-j)}-pH_{La+Ljp}),
\end{equation}
where $\mathbf B_{\mathbf N}(m):=\prod_{j=1}^k \mathbf B_{N_j}(m),$
the quantities $\mathbf B_{N_j}(m)$ being defined in 
\eqref{eq:Bzudilin}.
In view of Lemma~\ref{lem:4}, proving 
Theorem~\ref{thm:4} is equivalent to proving that 
\begin{equation} \label{eq:Cconga} 
\mathbf C(a+Kp) \in p \mathbb{Z}_p
\end{equation}
for all primes $p$ and non-negative integers $a$ and $K$ with $0\le a<p$.
Since 
$$
H_J = \sum_{j=1} ^{\fl{J/p}} \frac{1}{pj} + 
\underset{p\nmid j}{\sum_{j=1} ^{J}}\;
\frac{1}{j},
$$
we have
\begin{equation} \label{eq:J} 
pH_{J} \equiv H_{\fl{J/p}} \mod p\mathbb{Z}_p.
\end{equation}
Applying this with $J=La+Ljp$, we get
$$
pH_{La+Ljp}\equiv H_{\lfloor La/p\rfloor+Lj} \mod p\mathbb{Z}_p.
$$
This implies that
\begin{equation} \label{eq:Cequiv}
\mathbf C(a+Kp)  \equiv \sum_{j=0}^K 
\mathbf B_{\mathbf N}(a+jp)\mathbf B_{\mathbf N}(K-j) (H_{L(K-j)}-H_{\lfloor
La/p\rfloor+Lj}) \mod p \mathbb{Z}_p.
\end{equation}

We now want to transform the sum on the right-hand side 
of~\eqref{eq:Cequiv} to a more manageable expression. 
In particular, we want to get rid of the floor function $\fl{La/p}$.
In order to achieve this, we will prove the following lemma
in Section~\ref{sec:9}.

\begin{lem}\label{lem:12a}
For any prime $p$, non-negative integers $a$ and $j$ with $0\le a<p$, 
positive integers $N_1,N_2,\dots,N_k$, and $L\in\{1,2, \ldots, 
\max(N_1, \ldots, N_k)\}$, we have 
\begin{equation} \label{eq:congrconj1a} 
\mathbf{B}_{\mathbf{N}}(a+pj)\big(H_{Lj+ \lfloor La/p\rfloor}
-H_{Lj}\big) \in p\mathbb{Z}_p. 
\end{equation}
\end{lem}

It follows from Eq.~\eqref{eq:Cequiv} and Lemma~\ref{lem:12a} that 
 $$
\mathbf C(a+Kp) \equiv \sum_{j=0}^K \mathbf B_{\mathbf N}(a+jp)
\mathbf B_{\mathbf N}(K-j) \Big(H_{L(K-j)}-H_{Lj}\Big)
\mod p\mathbb{Z}_p\ ,
$$
which can be rewritten as 
\begin{equation}\label{eq:beforelemma4.2a}
\mathbf C(a+Kp) \equiv -\sum_{j=0}^K
H_{Lj}\Big(\mathbf B_{\mathbf N}(a+jp)\mathbf B_{\mathbf N}(K-j)-
\mathbf B_{\mathbf N}(j)\mathbf B_{\mathbf N}(a+(K-j)p)\Big) 
\mod p\mathbb{Z}_p\ .
\end{equation}

We now use a combinatorial lemma due to Dwork
(see~\cite[Lemma~4.2]{dwork}) which provides   
an alternative way to write the sum on the
right-hand side of~\eqref{eq:beforelemma4.2a}: namely, we have 
\begin{equation} \label{eq:107a}
\sum_{j=0}^K H_{Lj}\Big(\mathbf B_{\mathbf N}(a+jp)\mathbf B_{\mathbf
N}(K-j)-\mathbf B_{\mathbf N}(j)\mathbf B_{\mathbf N}(a+(K-j)p)\Big)
=
\sum _{s=0} ^{r}
\sum _{m=0} ^{p^{r+1-s}-1}\mathbf Y_{m,s},
\end{equation}
where $r$ is such that $K<p^r$, and 
\begin{equation*} 
\mathbf Y_{m,s}:=\big(H_{Lmp^s}-H_{L\fl{m/p}p^{s+1}}\big)\mathbf S(a,K,s,p,m),
\end{equation*}
the expression $\mathbf S(a,K,s,p,m)$ being defined by
$$
\mathbf S(a,K,s,p,m):=\sum _{j=mp^s} ^{(m+1)p^s-1}
\big(\mathbf B_{\mathbf N}(a+jp)\mathbf B_{\mathbf N}(K-j)
-\mathbf B_{\mathbf N}(j)\mathbf B_{\mathbf N}(a+(K-j)p)
\big).
$$
In this expression for $\mathbf{S}(a,K,s,p,m)$, it is assumed that
$\mathbf{B}_{\mathbf N}(n)=0$ for negative integers~$n.$ 

It would suffice to prove that 
\begin{equation} \label{eq:Yms} 
\mathbf Y_{m,s}\in p\mathbb Z_p
\end{equation}
for all $m$ and $s$, because~\eqref{eq:beforelemma4.2a} and \eqref{eq:107a}
would then imply that $\mathbf C(a+Kp)\in p\mathbb Z_p$, as desired.

We will prove \eqref{eq:Yms} in the following manner.
The expression for $\mathbf S(a,K,s,p,m)$ is of the form considered in 
Proposition~\ref{prop:dworkcongruence} 
in Section~\ref{sec:1}. 
The proposition will enable us to prove the
following fact in 
Section~\ref{sec:10}. 

\begin{lem} \label{lem:strat3} 
For all primes $p$ and
non-negative integers $a,m,s,K$ with $0\le a<p$, we have 
$$
\mathbf S(a,K,s,p,m) \in p^{s+1}\mathbf{B}_{\mathbf{N}}(m) \mathbb{Z}_p.
$$
\end{lem}

Furthermore, in Section~\ref{sec:6a} we shall prove the following lemma.

\begin{lem} \label{lem:strat4} 
For all primes $p$, non-negative integers $m$, 
positive integers $N_1,N_2,\dots,N_k$, and $L\in\{1,2, \ldots, 
\max(N_1, \ldots, N_k)\}$, we have
\begin{equation} \label{eq:110a} 
\mathbf{B}_{\mathbf{N}}(m) \left(H_{Lmp^s}-H_{L\lfloor m/p\rfloor
p^{s+1}}\right) \in \frac{1}{p^s} \,\mathbb{Z}_p. 
\end{equation}
\end{lem}
It is clear that Lemmas~\ref{lem:strat3} and~\ref{lem:strat4}
imply~\eqref{eq:Yms}. 
This completes the outline of the proof of Theorem~\ref{thm:4}.

\begin{remark} \label{foot:1}
To prove the refinement announced at the end of the Introduction
that\break 
${\mathbf q}_{1,\mathbf{N}}(z)^{1/\mathbf{B}_{\mathbf{N}}(1)}\in
\mathbb{Z}[[z]]$, by Lemma~\ref{lem:4} we should show that
$\mathbf C(a+Kp) \in p\mathbf B_{\mathbf N}(1) \mathbb{Z}_p$ instead
of the weaker \eqref{eq:Cconga}, 
which means to show that
\begin{equation} \label{eq:Cequiv2}
\mathbf C(a+Kp)  \equiv \sum_{j=0}^K 
\mathbf B_{\mathbf N}(a+jp)\mathbf B_{\mathbf N}(K-j) (H_{L(K-j)}-H_{\lfloor
La/p\rfloor+Lj}) \mod p \mathbf B_{\mathbf N}(1)\mathbb{Z}_p
\end{equation}
instead of the weaker \eqref{eq:Cequiv}, that 
$\mathbf{B}_{\mathbf{N}}(a+pj)\big(H_{j+ \lfloor a/p\rfloor}
-H_{j}\big) \in p\mathbf B_{\mathbf N}(1)\mathbb{Z}_p$ instead of 
\eqref{eq:congrconj1a} (but this is trivial because $H_{j+ \lfloor a/p\rfloor}
-H_{j}=0$), and that 
\begin{equation} \label{eq:BH}
\mathbf{B}_{\mathbf{N}}(m) \left(H_{Lmp^s}-H_{L\lfloor m/p\rfloor
p^{s+1}}\right) \in 
\frac{\mathbf B_{\mathbf N}(1)}{p^s} \,\mathbb{Z}_p
\end{equation}
instead of the
weaker \eqref{eq:110a}. To establish \eqref{eq:Cequiv2} one would
apply the same type of argument as the one establishing
\eqref{eq:firstreduction}, however with Lemma~\ref{lem:multinomial/N!}
replaced by Lemma~\ref{lem:diviBB},
the latter lemma being proved in Section~\ref{sec:8}. To prove \eqref{eq:BH}, 
Lemma~\ref{lem:diviBB} must be used in \eqref{eq:vBH}.
\end{remark}

\section{Outline of the proof of Theorem~\ref{thm:2}}\label{sec:2}

This section is devoted to an outline of the proof of
Theorem~\ref{thm:2}, reducing it to 
Lemmas~\ref{lem:12}--\ref{lem:11}. These lemmas are subsequently
proved in Sections~\ref{sec:4}--\ref{sec:6}.

We follow the strategy that we used in Section~\ref{sec:7} to prove
Theorem~\ref{thm:4}; that is, 
by the consequence of Dwork's Lemma given in
Lemma~\ref{lem:4}, we want to prove that  
$$F_{\mathbf N}(z)G_{L,\mathbf N}(z^p)-pF_{\mathbf N}(z^p)G_{L,\mathbf
N}(z) \in p
\frac{M_{\mathbf{N}}}{\Th_L} z \mathbb{Z}_p[[z]].$$ 

The $(a+Kp)$-th Taylor coefficient of
$F_{\mathbf N}(z)G_{L,\mathbf N}(z^p)-pF_{\mathbf N}(z^p)G_{L,\mathbf
N}(z)$ is  
\begin{equation} \label{eq:C} 
C(a+Kp):=\sum_{j=0}^K B_{\mathbf N}(a+jp)B_{\mathbf N}(K-j) (H_{L(K-j)}-pH_{La+Ljp}),
\end{equation}
where $B_{\mathbf N}(m)=\prod_{j=1}^k B_{N_j}(m)$ with
$B_{N}(m):=\frac{(Nm)!}{m!^{N}}$  
(not to be confused with $\mathbf{B}_{\mathbf N}(m)$ and
$\mathbf{B}_{N}(m)$).   
In view of Lemma~\ref{lem:4}, proving Theorem~\ref{thm:2} is equivalent to
proving that 
\begin{equation} \label{eq:Ccong} 
C(a+Kp) \in p \frac {M_{\mathbf{N}}} {\Th_L}\mathbb{Z}_p
\end{equation}
for all primes $p$ and non-negative integers $a$ and $K$ with $0\le a<p$.

The following simple lemma will be frequently used in the sequel.

\begin{lem} \label{lem:multinomial/N!}
For all integers $m\ge 1$ and $N\ge 1$, we have 
$$
B_N(m) \in N! \,\mathbb{Z}.
$$
\end{lem}
\begin{proof} Set $U_m(N)=\frac{(Nm)!}{m!^N N!}$. For any $m, N\ge 1$,
we have the trivial relation  
$$
U_m(N+1) = \binom{Nm+m-1}{m-1} U_m(N).
$$
Therefore, since $U_m(1)=1$, the result follows by induction on $N$.
\end{proof}
We deduce in particular that $B_{\mathbf N}(m)\in M_{\mathbf{N}}\mathbb{Z}$ for
any $m\ge 1$. 

\medskip

Using this together with \eqref{eq:J} 
specialised to $J=La+Ljp$, 
we infer
\begin{equation}\label{eq:firstreduction}
C(a+Kp) \equiv \sum_{j=0}^K B_{\mathbf N}(a+jp)B_{\mathbf N}(K-j) (H_{L(K-j)}-H_{\lfloor
La/p\rfloor+Lj}) \mod pM_{\mathbf{N}}\mathbb{Z}_p.
\end{equation}
Indeed, if $K\ge 1$ or $a\ge1$, this is because 
$a+jp$ and $K-j$ cannot be simultaneously zero and therefore at least one of 
$B_{\mathbf N}(a+jp)$ or $B_{\mathbf N}(K-j)$ is divisible by $M_{\mathbf{N}}$ by 
Lemma~\ref{lem:multinomial/N!}. In the 
remaining case $K=a=j=0$, we note that the difference of harmonic numbers
in~\eqref{eq:firstreduction} 
is equal to $0$, and therefore the congruence~\eqref{eq:firstreduction} 
holds trivially because $C(0)=0$.
\medskip  

The analogue of Lemma~\ref{lem:12a} in the present context, which allows
us to get rid of the floor function $\fl{La/p}$ and rearrange the sum
over $j$, is the following lemma. The proof can be found in
Section~\ref{sec:4}.

\begin{lem} \label{lem:12}
For any prime $p$, non-negative integers $a$ and $j$ with $0\le a<p$, 
positive integers $N_1,N_2,\dots,N_k$, and $L\in\{1,2, \ldots, 
\max(N_1, \ldots, N_k)\}$, we have 
\begin{equation} \label{eq:congrconj1}
B_{\mathbf N}(a+pj)\left(H_{Lj+\lfloor La/p\rfloor} - H_{Lj}\right) \in p\frac
{M_{\mathbf{N}}} {\Th_L}\mathbb{Z}_p.
\end{equation}
\end{lem}

We now do the same rearrangements as those after Lemma~\ref{lem:12a}
to conclude that
\begin{equation*}
C(a+Kp) \equiv 
-
\sum _{s=0} ^{r}
\sum _{m=0} ^{p^{r+1-s}-1}Y_{m,s}\mod p\frac
{M_{\mathbf{N}}} {\Th_L}\mathbb{Z}_p,
\end{equation*}
where $r$ is such that $K<p^r$, and 
\begin{equation*} 
Y_{m,s}:=\big(H_{Lmp^s}-H_{L\fl{m/p}p^{s+1}}\big)S(a,K,s,p,m),
\end{equation*}
the expression $S(a,K,s,p,m)$ being defined by
$$
S(a,K,s,p,m):=\sum _{j=mp^s} ^{(m+1)p^s-1}\big(B_{\mathbf N}(a+jp)
B_{\mathbf N}(K-j)-B_{\mathbf N}(j)B_{\mathbf N}(a+(K-j)p)
\big).
$$
In this expression for $S(a,K,s,p,m)$, it is assumed that $B_{\mathbf
N}(n)=0$ for negative integers~$n$. 

\medskip

If we prove that 
\begin{equation}
\label{eq:yms}
Y_{m,s}\in p\frac {M_{\mathbf{N}}} {\Th_L}\mathbb Z_p ,
\end{equation}
then 
$
C(a+Kp) \in  p\frac
{M_{\mathbf{N}}} {\Th_L}\mathbb{Z}_p, 
$
as desired. 

Now, the last assertion follows from the following two lemmas.
Lemma~\ref{lem:10} is the special case 
of Lemma~\ref{lem:strat3} where 
the vector $\mathbf N$ is specialised in the way described at the
beginning of Section~\ref{sec:1.3} (in which case the quantity
$\mathbf{B}_{\mathbf N}(m)$ of Lemma~\ref{lem:strat3} reduces to
$B_{\mathbf N}(m)$, and hence $\mathbf S(a,K,s,p,m)$ to
$S(a,K,s,p,m)$).
On the other hand, Lemma~\ref{lem:11} is the analogue of
Lemma~\ref{lem:strat4}. Its proof can be found in
Section~\ref{sec:6}.

\begin{lem} \label{lem:10}
For all primes $p$ and
non-negative integers $a,m,s,K$ with $0\le a<p$, we have 
\begin{equation}
\label{eq:congruenceS}
S(a,K,s,p,m)\in p^{s+1} B_{\mathbf N}(m) \mathbb{Z}_p.
\end{equation}
\end{lem}

\begin{lem} \label{lem:11}
For all primes $p$, non-negative integers $m$, 
positive integers $N_1,N_2,\dots,N_k$, and $L\in\{1,2, \ldots, 
\max(N_1, \ldots, N_k)\}$, we have
\begin{equation} 
\label{eq:110}
B_{\mathbf N}(m)\big(H_{Lmp^s}-H_{L\fl{m/p}p^{s+1}}\big)\in \frac {M_{\mathbf{N}}}
{p^s\Th_L}\mathbb Z_p\ .
\end{equation}
\end{lem}

It is clear that~\eqref{eq:congruenceS} and~\eqref{eq:110}
imply~\eqref{eq:yms}. 
This completes the outline of the proof of Theorem~\ref{thm:2}.

\section{Dwork's theory of formal congruences} \label{sec:1}

In this section, we review those aspects of Dwork's theory on
which the arguments of the proofs of Theorems~\ref{thm:4} and
\ref{thm:2} (see Sections~\ref{sec:7} and \ref{sec:2}) are based.

\medskip
First, consider a formal power series $S(z)\in \mathbb{Q}[[z]]$ and suppose 
that we want to prove that $S(z)\in\mathbb{Z}[[z]]$. 

\begin{lem} \label{lem:1} 
Let $S(z)$ be a power series in $\mathbb{Q}[[z]]$. If 
$S(z)\in\mathbb{Z}_p[[z]]$ for any prime number $p$, then
$S(z)\in\mathbb{Z}[[z]]$. 
\end{lem}
This is a consequence of the fact that, given $x\in \mathbb{Q}$, we have  
$x\in \mathbb{Z}$ if and only if 
$x\in \mathbb{Z}_p$ for all prime numbers $p$.
Hence we can work in $\mathbb{Q}_p$ for any fixed prime $p$.

\begin{lem}[\sc ``Dwork's Lemma''] \label{lem:2} 
Let $S(z)\in 1+z\mathbb{Q}_p[[z]]$. Then, 
we have $S(z)\in 1+z\mathbb{Z}_p[[z]]$ if and only if 
$$
\frac{S(z^p)}{S(z)^p }
\in 1+p z\mathbb{Z}_p[[z]].
$$
\end{lem}
\begin{proof} The proof is neither difficult nor long and can 
for example be found in the book of
Lang~\cite[Ch.~14, p.~76]{lang}. Lang attributes this lemma to 
Dieudonn\'e and Dwork.
\end{proof}

We now suppose that $S(z)=\exp(T(z)/\tau)$ for some $T(z)\in
z\mathbb{Q}[[z]]$ and some integer $\tau\ge 1$.
Dwork's Lemma implies the following result: 
$\tau$ being any fixed positive integer, 
we have $\exp\big(T(z)/ \tau \big) \in 1+z\mathbb{Z}_p[[z]]$ if 
and only if $T(z^p)-pT(z)\in p \tau z\mathbb{Z}_p[[z]]$.
(See~\cite[Corollary~6.7]{lianyau} for a proof.)
Since we will be interested in the case when $T(z)=g(z)/f(z)$ with
$f(z)\in 1+z\mathbb{Z}[[z]]$ and $g(z)\in z\mathbb{Q}[[z]]$, 
we state this result as follows.

\begin{lem}\label{lem:4} Given two formal power series 
$f(z)\in 1+z\mathbb{Z}[[z]]$ and $g(z)\in z\mathbb{Q}[[z]]$ and an
integer $\tau\ge 1$, we have 
$\exp\big(g(z)/(\tau f(z))\big) \in 1+z\mathbb{Z}_p[[z]]$ if and only if 
\begin{equation}\label{eq:uv=pvu}
f(z)g(z^p)-p\,f(z^p)g(z)\in p\tau z\mathbb{Z}_p[[z]].
\end{equation}
\end{lem}

Because of the special form of the functions which will play the role of 
$f(z)$ and $g(z)$, we 
will be able to deduce~\eqref{eq:uv=pvu} from the following crucial
result, also  
due to Dwork (see~\cite[Theorem~1.1]{dwork}). 

\begin{prop}[\sc ``Dwork's Formal Congruences Theorem''] 
\label{prop:dworkcongruence}
Let $A : \mathbb{Z}_{\ge 0}\to
\mathbb{Q}_p^{ \times}$, $g : \mathbb{Z}_{\ge 0}  
\to \mathbb{Z}_p\setminus\{0\}$ be mappings such that 

$(i)$ $\vert A(0)\vert_p =1$; 

$(ii)$ $A(m) \in g(m)\mathbb{Z}_p$;

$(iii)$ for all integers 
$u,v,n,s\ge 0$ such that $0\le u<p^s$ and 
$0 \le v<p$, 
we have 
\begin{equation} \label{eq:DworkA}
\frac{A(v+up +np^{s+1})}{A(v+up)}-\frac{A(u
+np^{s})}{A(u)} \in p^{s+1} \frac{g(n)}{g(v+up)}\,
\mathbb{Z}_p. 
\end{equation}
Furthermore, let $F(z) = \sum_{m=0}^{\infty} A(m) z^m$, and
$$
F_{m,s}(z)=\sum_{j=mp^s}^{(m+1)p^s-1} A(j)z^j.
$$
Then, for any integers  $m, s \ge 0$, we have 
\begin{equation}\label{eq:newcongruence}
F(z^p)F_{m,s+1}(z) - F(z)F_{m,s}(z^p)\in p^{s+1} g(m) \mathbb{Z}_p[[z]],
\end{equation}
or, equivalently, 
\begin{equation}\label{eq:newcongruence2}
\sum _{j=mp^s} ^{(m+1)p^s-1}\big(A(a+jp)A(K-j)-A(j)A(a+(K-j)p)
\big)\in p^{s+1} g(m) \mathbb{Z}_p
\end{equation}
for all $a$ and $K$ with $0\le a<p$ and $K\ge0$.
\end{prop}

\begin{Remarks} 
(a) Dwork's original theorem is in fact more general in that it
contains families of functions $A_r$ and $g_r$, $r=0,1,2,\dots,$
(which are all equal to $A$, respectively to $g$, in the above
specialisation). Moreover,
Dwork proved his theorem with 
$A_r:\mathbb{Z}_{\ge 0}\to\mathbb{C}_p^{\times}$, 
$g_r:\mathbb{Z}_{\ge 0}\to\mathcal{O}_p\setminus\{0\}$ and 
$A_r(m) \in g_r(m)\mathcal{O}_p$ 
(where $\mathcal{O}_p$ is the ring of integers in $\mathbb{C}_p$).
He obtained
a result similar to~\eqref{eq:newcongruence} and \eqref{eq:newcongruence2}, 
with $\mathcal{O}_p$ instead of
$\mathbb{Z}_p$. In our more restrictive setting,~\eqref{eq:newcongruence} 
and \eqref{eq:newcongruence2} 
hold because $\big(p^{s+1}g(m)\mathcal{O}_p\big) \cap \mathbb{Q}_p = 
p^{s+1}g(m)\mathbb{Z}_p.$

(b) For any integers  $a$ and $K$ with $0\le a<p$ and
$K\ge0$, the sum  
\begin{equation}\label{eq:rajout1}
\sum _{j=mp^s} ^{(m+1)p^s-1}\big(A(a+jp)A(K-j)-A(j)A(a+(K-j)p)
\big)
\end{equation}
is exactly the $(a+pK)$-th Taylor coefficient of 
$F(z^p)F_{m,s+1}(z) - F(z)F_{m,s}(z^p)$, 
which explains the equivalence between the formal
congruence~\eqref{eq:newcongruence} and the  
congruence~\eqref{eq:newcongruence2}. Note that in~\eqref{eq:rajout1} the 
value of $A$ at negative integers 
must be taken as $0$.

(c) Most authors chose $g(m)=1$ or a constant in $m$. We will
use instead $g(m)=A(m)$: this choice has already been made by Dwork
in~\cite[Sec.~2, p.~37]{dworkihes}.

(d) Dwork also applied his methods to the problems 
considered in the present paper. Indeed, he proved a 
result,
namely \cite[p.~311, Theorem~4.1]{dwork}, which implies that for any
prime $p$ that does not divide $N_1N_2\cdots N_k$, 
the mirror maps ${\mathbf q}_{{\mathbf N}}(z)$ have Taylor 
coefficients in $\mathbb{Z}_p$ (see~\cite[Proposition 2]{zud} for
details). This fact was used by the  
authors of~\cite{lianyau1, lianyau, zud} who 
focussed essentially 
on the remaining case when $p$ divides $N_1N_2\cdots N_k$. Our
approach is different, for we make no distinction of  
this kind between prime numbers.
\end{Remarks}

During the proofs of Lemma~\ref{lem:strat3} 
in Section~\ref{sec:10}, 
we will
also use certain properties of the $p$-adic gamma 
function $\Gamma_p$. This function is defined on integers $n\ge 1$ by 
$$
\Gamma_p(n) = (-1)^n \prod_{\stackrel{k=1}{(k,p)=1}}^{n-1} k.
$$
We will not consider its extension to 
$\mathbb{Z}_p.$
In the following lemma,
we collect the results on $\Gamma_p$ that we shall need later on.
\begin{lem}\label{lem:gammap}
$(i)$ For all integers $n\ge 1$, we have 
\begin{equation*}
\frac{(np)!}{n!} = (-1)^{np+1}p^n \Gamma_p(1+np).
\end{equation*}
$(ii)$ For all integers $k\ge 1,n\ge 1,s\ge 0$, we have 
\begin{equation*}
\Gamma_p(k+np^s) \equiv \Gamma_p(k) \mod p^s.
\end{equation*}
\end{lem}
\begin{proof} 
See~\cite[Lemma~7]{zud} for $(i)$ and~\cite[p.~71, Lemma~1.1]{lang} for $(ii)$.
\end{proof}

\section{Auxiliary lemmas}
\label{sec:8}

In this section, we establish three
auxiliary results.  
The first one, Lemma~\ref{lem:10a}, is required for the proof of
Lemma~\ref{lem:12a} in 
Section~\ref{sec:9}, while the second one, Lemma~\ref{lem:ultime},
is required for the proof of Lemma~\ref{lem:strat3} in
Section~\ref{sec:10}. The third result, Lemma~\ref{lem:diviBB}, 
justifies an assertion made in the Introduction
(see item~(b) in the remarks after Theorem~\ref{thm:2}). 
Moreover, the proofs of Lemmas~\ref{lem:ultime} and~\ref{lem:diviBB}
make themselves use of Lemma~\ref{lem:10a}.

\begin{lem} \label{lem:10a}
For any integer $N\ge 1$ with associated parameters $\al_i, \beta_i,
\mu, \eta$, the function 
$$\Delta(x):=
 \sum_{i=1}^{\mu}
\left\lfloor \al_{i}x \right\rfloor - \sum_{i=1}^{\eta} 
\left\lfloor \be_{i}x\right\rfloor
$$
has the following properties:
\begin{enumerate}
\item [$(i)$] $\De$ is $1$-periodic. In particular, $\De(n)=0$ for
all integers $n$.
\item [$(ii)$] For all integers $n$, $\De$ is weakly increasing on
intervals of the form $[n,n+1)$.
\item [$(iii)$] For all real numbers $x$, we have $\De(x)\ge0$.
\item [$(iv)$] For all rational numbers $r\neq 0$ whose denominator 
is an element of 
$\{2,3,\dots,N\}$, 
we have $\Delta(r)\ge1$.
\end{enumerate}
\end{lem}

\begin{Remark} 
Clearly, the function $\Delta$ is a step function.
The proof below shows that, in fact, all the jumps of $\Delta$ 
at non-integral places have the 
value $+1$ and occur exactly at rational numbers of the form 
$r/N$, with $r$ coprime to $N$. 
\end{Remark} 

\begin{proof}
Property~$(i)$ follows from 
the equality $\sum_{i=1}^\mu \al_i = \sum_{i=1}^\eta \beta_i$
and the trivial fact that $\De(0)=0$.

We turn our attention to property~$(ii)$.
For convenience of notation, let
$$N=p_1^{e_1}p_2^{e_2}\cdots p_\ell ^{e_\ell }$$
be the prime factorisation of $N$, where,
as before, $p_1,p_2,\dots,p_\ell $ are the
distinct prime factors of $N$, and where $e_1,e_2,\dots,e_\ell $ are
positive integers. 

As we already observed in the remark above, 
the function $\Delta$ is a step function. Moreover,
jumps of $\Delta$ can only occur at values of $x$
where some of the $\al_ix$, $1\le i\le\mu$, or some of the $\be_jx$,
$1\le j\le \eta$, (or both) are integers.
At these values of $x$, the value of a (possible) jump is the difference
between the number of $i$'s for which $\al_ix$ is integral and the number
of $j$'s for which $\be_jx$ is integral. In symbols, the value of the
jump is
\begin{equation} \label{eq:diff} 
\#\{i:1\le i\le\mu\text{ and }\al_ix\in\mathbb Z\}
-\#\{j:1\le j\le\eta\text{ and }\be_jx\in\mathbb Z\}.
\end{equation}

Let $X$ be
the place of a jump, $X$ not being an integer. Then we can write $X$
as
$$X=\frac {Z} {p_1^{f_1}p_2^{f_2}\cdots p_\ell ^{f_\ell }},$$
where $f_1,f_2,\dots,f_\ell$ are non-negative integers, not all zero,
and where $Z$ is a non-zero integer relatively prime to
${p_1^{f_1}p_2^{f_2}\cdots p_\ell ^{f_\ell }}$.
Given
$$\al_i=p_1^{a_1}p_2^{a_2}\cdots p_\ell ^{a_\ell }$$
with $e_1+e_2+\dots+e_\ell -(a_1+a_2+\dots+a_\ell )$ even
and $0\le e_k-a_k\le 1$ for each $k=1, 2, \ldots, \ell $,
the number $\al_iX$ will be integral if and only if $a_k\ge f_k$ for all
$k\in\{1,2,\dots,\ell \}$. Similarly, given
$$\be_j=p_1^{b_1}p_2^{b_2}\cdots p_\ell ^{b_\ell }$$
with $e_1+e_2+\dots+e_\ell -(b_1+b_2+\dots+b_\ell )$ odd
and $0\le e_k-b_k\le 1$ for each $k=1, 2, \ldots, \ell $,
the number 
$\be_jX$ will be integral if and only if $b_k\ge f_k$ for all
$k\in\{1,2,\dots,\ell \}$. 
We do not have to take into account the $\be_j$'s which are $1$,
because $1\cdot X=X$ is not an integer by assumption.
For the generating function of vectors
$(c_1,c_2,\dots,c_\ell )$ with $e_k\ge c_k\ge f_k$ 
and $e_{k} - c_{k}\le 1$, 
we have
$$
\sum _{c_1=\max\{e_1-1, f_1\}} ^{e_1}\cdots
\sum _{c_\ell =\max\{e_\ell -1, f_\ell \}} ^{e_\ell }
z^{e_1+\dots+e_\ell -(c_1+\dots+c_\ell )}=
\prod _{k=1} ^{\ell }\left(1+z\cdot\min\{1,e_k-f_k\}\right).
$$
We obtain the difference in \eqref{eq:diff} (with $X$ in place of $x$) 
by putting $z=-1$ on the
left-hand side of this relation. The product on the right-hand side
tells us that this difference is $0$ in case that 
$e_k\ne f_k$ for some $k$, while it is $1$ otherwise. Thus, all the jumps 
of the function $\Delta$ at non-integral places have the value $+1$.

Property~$(iii)$ follows now easily from $(i)$ and $(ii)$.

In order to prove 
$(iv)$, we observe that the first jump of $\Delta$ in the interval
$[0,1)$ occurs at $x=1/N$. Thus, $\Delta(x)\ge1$ for all $x$ in $[1/N,1)$.
This implies in particular that $\Delta(r)\ge1$ for all the above rational
numbers $r$ in the interval $[1/N,1)$. That the same assertion holds
in fact for {\it all\/} the above rational numbers $r$ (not
necessarily restricted to $[1/N,1)$) follows now from the
1-periodicity of the function $\Delta$.
\end{proof}

\begin{lem} \label{lem:ultime} 
For any integers $m,r,w\ge 0$ such that
$0\le w<p^r$, we have  
\begin{equation} \label{eq:ultime}
\frac{\mathbf{B}_{\mathbf{N}}(w+mp^{r})}{\mathbf{B}_{\mathbf{N}}(m)}
\in \mathbb{Z}_p,
\end{equation}
where $\mathbf{B}_{\mathbf{N}}(m)$ is the quantity defined after
\eqref{eq:firstreductiona}.
\end{lem}

\begin{proof}
We first show that we can assume that $m$ is coprime to
$p$. Indeed, let us write $m=hp^t$ with $\gcd(h,p)=1$. We have to prove
that  
$$
\frac{\mathbf{B}_{\mathbf{N}}(w+hp^{r+t})}{\mathbf{B}_{\mathbf{N}}(hp^t)}
\in \mathbb{Z}_p. 
$$
Since
$v_p\big(\mathbf{B}_{\mathbf{N}}(hp^t)/\mathbf{B}_{\mathbf{N}}(h)
\big)=0$
(as can be easily seen from \eqref{eq:Bzudilin} and Legendre's
formula
$
v_p(n!) = \sum_{k=1}^{\infty} \lfloor \frac{n}{p^k} \rfloor
$),
this amounts to prove 
that 
$$
\frac{\mathbf{B}_{\mathbf{N}}(w+hp^{r+t})}{\mathbf{B}_{\mathbf{N}}(h)}
\in \mathbb{Z}_p, 
$$
which is the content of the lemma with $r+t$ instead of $r$ and $h$
instead of $m$, with  
$w<p^r<p^{r+t}$.

Therefore, from now on, we assume that $\gcd(m,p)=1$ (however, 
this assumption will only be used after~\eqref{eq:ultime3}). Since
$v_p\big(\mathbf{B}_{\mathbf{N}}(mp^r)/\mathbf{B}_{\mathbf{N}}(m)
\big)=0$, we have 
to prove that 
$$
v_p\left(
\frac{\mathbf{B}_{\mathbf{N}}(w+mp^{r})}{\mathbf{B}_{\mathbf{N}}(mp^r)}\right)
\ge 0 
$$
or, in an equivalent form, that
\begin{multline}\label{eq:ultime2}
\sum_{j=1}^k\sum_{\ell=1}^{\infty} \Bigg( \bigg(\sum_{i=1}^{\mu_j} 
 \left\lfloor \frac{\al_{i,j}(w+mp^r)}{p^\ell}\right\rfloor -
\sum_{i=1}^{\eta_j}  \left\lfloor
\frac{\be_{i,j}(w+mp^r)}{p^\ell}\right\rfloor 
\bigg)
\\
- 
\bigg( \sum_{i=1}^{\mu_j} \left\lfloor \frac{\al_{i,j}
mp^r}{p^\ell}\right\rfloor - \sum_{i=1}^{\eta_j}  \left\lfloor  
\frac{\be_{i,j} mp^r}{p^\ell}\right\rfloor\bigg)
\Bigg) \ge 0,
\end{multline}
where $\al_{i,j}, \beta_{i,j}, \mu_j, \eta_j$ are the parameters
associated to $N_j$.

If $\ell\le r$, then for any $j\in\{1, 2,\ldots, k \}$,  
$$
\sum_{i=1}^{\mu_j} \left\lfloor \frac{\al_{i,j} mp^r}{p^\ell}\right\rfloor 
- \sum_{i=1}^{\eta_j}  \left\lfloor \frac{\be_{i,j}
mp^r}{p^\ell}\right\rfloor   
= mp^{r-\ell} \left(\sum_{i=1}^{\mu_j} \al_{i,j} - \sum_{i=1}^{\eta_j}
\be_{i,j} \right)=0. 
$$
Moreover, 
$$
\sum_{i=1}^{\mu_j} 
 \left\lfloor \frac{\al_{i,j}(w+mp^r)}{p^\ell}\right\rfloor -
\sum_{i=1}^{\eta_j}  \left\lfloor
\frac{\be_{i,j}(w+mp^r)}{p^\ell}\right\rfloor \ge0
$$
because of Lemma~\ref{lem:10a}$(iii)$ with $N=N_j$. It therefore
suffices to show
\begin{multline} \label{eq:ells+1}
\sum_{j=1}^k\sum_{\ell=r+1}^{\infty} \Bigg( \bigg(\sum_{i=1}^{\mu_j} 
 \left\lfloor \frac{\al_{i,j}(w+mp^r)}{p^\ell}\right\rfloor -
\sum_{i=1}^{\eta_j}  \left\lfloor
\frac{\be_{i,j}(w+mp^r)}{p^\ell}\right\rfloor 
\bigg)
\\
- 
\bigg( \sum_{i=1}^{\mu_j} \left\lfloor \frac{\al_{i,j}
mp^r}{p^\ell}\right\rfloor - \sum_{i=1}^{\eta_j}  \left\lfloor  
\frac{\be_{i,j} mp^r}{p^\ell}\right\rfloor\bigg)
\Bigg) \ge 0,
\end{multline}
(The reader should note the difference to \eqref{eq:ultime2}
occurring in the summation bounds for $\ell$.)
For $\ell>r$, set $x_{\ell}=\{mp^r/p^{\ell}\}$
and  $y_{\ell}=\{(w+mp^r)/p^{\ell}\}$. Using again 
$\sum _{i=1} ^{\mu_j}\al_{i,j}-\sum _{i=1} ^{\eta_j}\be_{i,j}=0$,
we see that the left-hand side of~\eqref{eq:ells+1} is equal to 
\begin{equation} \label{eq:ultime3} 
\sum_{j=1}^k \sum_{\ell=r+1}^{\infty} \Bigg( \bigg(\sum_{i=1}^{\mu_j} 
 \left\lfloor \al_{i,j} y_{\ell}\right\rfloor - \sum_{i=1}^{\eta_j}
\left\lfloor \be_{i,j} y_{\ell} \right\rfloor \bigg) 
\\
- 
\bigg( \sum_{i=1}^{\mu_j} \left\lfloor \al_{i,j} x_{\ell}\right\rfloor
- \sum_{i=1}^{\eta_j}  \left\lfloor \be_{i,j} 
x_{\ell}\right\rfloor\bigg)\Bigg).
\end{equation}

We now claim that $x_{\ell}\le y_{\ell}$ for $\ell>r$.
To see this, we begin by the observation that, 
since $m$ and $p$ are coprime and $\ell>r$, the rational number
$m/p^{\ell-r}$ is not an integer. It follows that 
$$
x_{\ell} + \frac{1}{p^{\ell-r}}= \left\{\frac{m}{p^{\ell-r}}\right\}
+ \frac{1}{p^{\ell-r}} \le 1. 
$$
Hence, since $w<p^r$, we infer that 
$$
x_{\ell}+ \frac{w}{p^\ell} <1.
$$
On the other hand, we have
$$
y_{\ell}=\left\{\frac{w}{p^\ell} + \left\lfloor \frac{m}{p^{\ell-r}}
\right\rfloor+ x_{\ell}\right\} =  
\left\{\frac{w}{p^\ell} +  x_{\ell}\right\}=\frac{w}{p^\ell} +  x_{\ell}.
$$
Since $w\ge 0$, we obtain indeed $y_{\ell} \ge x_{\ell}$, as we claimed.

Using $x_{\ell}\le y_{\ell}$ together with Lemma~\ref{lem:10a}$(ii)$, 
we see that, for $\ell>r$ and $j=1,2,\ldots k$, 
we have 
$$
 \sum_{i=1}^{\mu_j} 
 \left\lfloor \al_{i,j} y_{\ell}\right\rfloor - \sum_{i=1}^{\eta_j}
\left\lfloor \be_{i,j} y_{\ell} \right\rfloor   
\ge 
 \sum_{i=1}^{\mu_j} \left\lfloor \al_{i,j} x_{\ell}\right\rfloor -
\sum_{i=1}^{\eta_j}  \left\lfloor \be_{i,j} x_{\ell}\right\rfloor,  
$$
which shows that the expression in \eqref{eq:ultime3} is non-negative,
thus establishing~\eqref{eq:ells+1} and also \eqref{eq:ultime2}. 
This finishes the proof of the lemma.
\end{proof}

We conclude this section with a result which was announced 
in item~(b) of the remarks after Theorem~\ref{thm:2}.
It is used nowhere else, but we give it 
here for the sake of completeness. 
It is a generalisation of Lemma~\ref{lem:multinomial/N!}. By the same
techniques used to prove Theorem~\ref{thm:2},  
it enables one to prove that 
${\mathbf q}_{1,\mathbf{N}}(z)^{1/\mathbf{B}_{\mathbf{N}}(1)}\in
\mathbb{Z}[[z]]$ (see Remark~\ref{foot:1} 
in Section~\ref{sec:7}). 

\begin{lem}\label{lem:diviBB} 
For any vector $\mathbf{N}$ and any integer $m \ge 1$, we have that 
$\mathbf{B}_{\mathbf{N}}(1)$ divides $\mathbf{B}_{\mathbf{N}}(m)$,
where $\mathbf{B}_{\mathbf{N}}(m)$ is the quantity defined after
\eqref{eq:firstreductiona}.
\end{lem}
\begin{proof} Obviously, it is sufficient to prove the assertion for
$k=1$ and $\mathbf{N}=(N)$. Let  
$\Delta$ be the function associated to $N$ as defined in Lemma~\ref{lem:10a}.
We want to prove that, for any prime $p$, we have
$v_p(\mathbf{B}_{N}(m))\ge v_p(\mathbf{B}_{N}(1))$. 
We can assume that $m$ and $p$ are coprime because 
$v_p(\mathbf{B}_{N}(mp^t))=v_p(\mathbf{B}_{N}(m))$ for any 
integers $m,t\ge 0$. 

Now, when $\gcd(m,p)=1$, we have that 
\begin{align*}
v_p(\mathbf{B}_{N}(m))& = \sum_{\ell=1}^{\infty} \Delta(m/p^\ell)
= \sum_{\ell=1}^{\infty} \Delta(\{m/p^\ell\})
\\
&\ge \sum_{\ell=1}^{\infty} \Delta(1/p^\ell) = v_p(\mathbf{B}_{N}(1)).
\end{align*}
Here,
we used the $1$-periodicity of $\Delta$ for the second equality.
For the inequality, we used that 
$\{m/p^\ell\} \ge 1/p^\ell$ (because $\gcd(m,p)=1$ 
implies that $m/p^\ell$ is not an integer) and the 
(partial) monotonicity of $\De$ described in Lemma~\ref{lem:10a}$(ii)$.
\end{proof}

\section{Proof of Lemma~\ref{lem:12a}}
\label{sec:9}

The assertion is trivially true if $\lfloor La/p\rfloor=0$, that is,
if $0\le a<p/L$. We may hence assume that $p/L\le a<p$ from now on.

We write
$\al_{i,m}$, $\beta_{i,m}$, $\mu_m$, and $\eta_m$ for
the parameters associated to $N_m$, $m=1,2,\dots,k$. We may assume 
that, without loss of generality,
$\max(N_1, \ldots, N_k)=N_k$.  
Then, using again Lemma~\ref{lem:10a}$(iii)$,
\begin{align*}
v_p\big(\mathbf{B}_{\mathbf{N}}(a+pj)\big) & =
\sum_{m=1}^k\sum_{\ell=1}^{\infty}\left( \sum_{i=1}^{\mu_m} 
\left\lfloor \frac{\al_{i,m}(a+pj)}{p^\ell}\right\rfloor - \sum_{i=1}^{\eta_m}
\left\lfloor \frac{\be_{i,m}(a+pj)}{p^\ell}\right\rfloor\right)\\
&\ge \sum_{\ell=1}^{\infty}\left( \sum_{i=1}^{\mu_k}
\left\lfloor \frac{\al_{i,k}(a+pj)}{p^\ell}\right\rfloor - \sum_{i=1}^{\eta_k} 
\left\lfloor
\frac{\be_{i,k}(a+pj)}{p^\ell}\right\rfloor\right)=\sum_{\ell=1}^{\infty}
\Delta_k\!\left(\frac{a+jp}{p^\ell}\right) 
\end{align*}
with 
$$
\Delta_k(x) : = \sum_{i=1}^{\mu_k}
\left\lfloor \al_{i,k}x \right\rfloor - \sum_{i=1}^{\eta_k} 
\left\lfloor \be_{i,k}x\right\rfloor.
$$

On the other hand, by definition of the harmonic numbers, we have
$$
H_{Lj+\lfloor La/p\rfloor} - H_{Lj}=\frac {1} {Lj+1}+\frac {1} {Lj+2}+\dots+
\frac {1} {Lj+\lfloor La/p\rfloor}.
$$
It therefore suffices to show that
\begin{equation}\label{eq:rajout8}
v_p\big(\mathbf{B}_{\mathbf{N}}(a+pj)\big) \ge
1+ v_p(Lj+\ep)
\end{equation} 
for any integer $\ep$ such that $1\le \ep\le \lfloor La/p\rfloor.$  
We have 
$$
\frac{a+jp}{p^\ell} = \frac{a-p\ep/L}{p^\ell} + \frac{pj+ p\ep/L }{p^\ell}.
$$

\subsection{First step} We claim that 
\begin{equation} \label{eq:claim1}
\Delta_k\!\left(\frac{a+jp}{p^\ell}\right) \ge  \Delta_k\!\left(\frac{pj+
p\ep/L }{p^\ell}\right). 
\end{equation}
To see this, we first observe that 
$$
\Delta_k\!\left(\frac{a+jp}{p^\ell}\right)=\Delta_k\!\left(\frac{a-p\ep/L}{p^\ell}
+ \frac{pj+ p\ep/L }{p^\ell}\right)  
= \Delta_k\!\left(\frac{a-p\ep/L}{p^\ell} + \left\{\frac{pj+ p\ep/L
}{p^\ell}\right\}\right) 
$$ 
because $\Delta_k$ is $1$-periodic. 

We now claim that 
\begin{equation}\label{eq:inequalities}
0\le \frac{a-p\ep/L}{p^\ell} + \left\{\frac{pj+ p\ep/L }{p^\ell}\right\} < 1.
\end{equation}
Indeed, positivity is clear and we now concentrate on the upper bound.  
We write $j=up^{\ell-1}+v$ with $0\le
v<p^{\ell-1}$. Hence,
$$
 \left\{\frac{pj+ p\ep/L }{p^\ell}\right\}  =  \left\{u+ \frac{pv+
p\ep/L }{p^\ell}\right\} =  
\left\{\frac{v}{p^{\ell-1}} + \frac{p\ep/L }{p^\ell}\right\}.
$$
Since $0 \le \ep \le \lfloor La/p\rfloor<L$, we have  $0\le
\frac{p\ep/L }{p^\ell} <1/p^{\ell-1}$ and therefore 
$$
0\le \frac{v}{p^{\ell-1}} + \frac{p\ep/L }{p^\ell} <
\frac{v}{p^{\ell-1}} + \frac{1}{p^{\ell-1}} \le 1 
$$
(where the last inequality holds by definition of $v$), whence
$$
\left\{\frac{pj+ p\ep/L }{p^\ell}\right\} = \frac{pv+ p\ep/L }{p^\ell}.
$$
Therefore, we have 
$$
 \frac{a-p\ep/L}{p^\ell} + \left\{\frac{pj+ p\ep/L }{p^\ell}\right\} =
\frac{a-p\ep/L}{p^\ell} +  \frac{pv+ p\ep/L }{p^\ell}  
= \frac{a}{p^\ell} + \frac{v}{p^{\ell-1}}.
$$
Since $\frac{v}{p^{\ell-1}}<1$ and $a<p$, we necessarily have  
$$
\frac{a}{p^\ell} + \frac{v}{p^{\ell-1}} < 1,
$$
as desired.

Since $\frac{a-p\ep/L}{p^\ell}\ge 0$, it follows from 
Lemma~\ref{lem:10a}$(i)$,$(ii)$ (with $\Delta=\Delta_k$) and
\eqref{eq:inequalities} that  
$$
\Delta_k\!\left(\frac{a-p\ep/L}{p^\ell} + \left\{\frac{pj+ p\ep/L
}{p^\ell}\right\}\right) \ge  
\Delta_k\!\left( \left\{\frac{pj+ p\ep/L }{p^\ell}\right\} \right) =
\Delta_k\!\left( \frac{pj+ p\ep/L }{p^\ell} \right). 
$$
Thus, we have proved the claim~\eqref{eq:claim1} 
made at the beginning of this step.

\subsection{Second step} Let us write $Lj+\ep=\beta p^d$,
where $d=v_p(Lj+\ep)$, so that  
$$
\frac{pj+ p\ep/L }{p^\ell} = \frac{\beta p^{d+1-\ell}}{L}.
$$
We have proved in the first step that 
\begin{equation}\label{eq:rajout7}
v_p\big(\mathbf{B}_{\mathbf{N}}(a+pj)\big) \ge \sum_{\ell=1}^\infty
\Delta_k\!\left(\frac{\beta p^{d+1-\ell}}{L}\right). 
\end{equation}
Now we claim that $\be p^{d+1-\ell }/L$ cannot be an integer.
Indeed, if it were, then $L\ga p^{\ell -1}=\be p^d=Lj+\ep$ for a suitable
integer $\ga$. It would follow that $L$ divides $\ep$, contradicting
$1\le \ep\le La/p<L$.
Furthermore, for  
$\ell\le d+1$, the denominator of $\frac{\beta p^{d+1-\ell}}{L}$ is
obviously at most $L$. 
Since $L\le N_k$, it follows then from
Lemma~\ref{lem:10a}$(iv)$, again with $\Delta=\Delta_k$,
that $\Delta_k(\beta p^{d+1-\ell}/L) \ge 1$ for any $\ell$ in
$\{1, 2, \ldots,  d+1\}.$  
Use of this estimation in \eqref{eq:rajout7} gives
$$v_p\big(\mathbf{B}_{\mathbf N}(a+pj)\big)\ge d+1=1+v_p(Lj+\ep).$$
This completes the proof of~\eqref{eq:rajout8} and, hence, of
Lemma~\ref{lem:12a}.

\section{Proof of Lemma~\ref{lem:strat3}}
\label{sec:10}

We want to use Proposition~\ref{prop:dworkcongruence}
with $A(m)=g(m)=\mathbf B_{\mathbf N}(m)$.
Clearly, the proposition would imply
that $\mathbf S(a,K,s,p,m)\in p^{s+1}
\mathbf B_{\mathbf N}(m)\mathbb Z_p$, and, thus, the claim.
So, we need to
verify the conditions $(i)$--$(iii)$ in the statement of the proposition.

Condition~$(i)$ is true since $\mathbf B_{\mathbf N}(0)=1$. 
Condition~$(ii)$ holds by the definitions of $A(m)$ and $g(m)$.
To check that Condition~$(iii)$ holds is more 
involved.
The proof will be decomposed in three steps. 
The reader should recall that
$$
\mathbf{B}_{\mathbf{N}}(m) := \prod_{j=1}^k \mathbf{B}_{N_j}(m),
$$
where $\mathbf B_{N_j}(m)$ is given by 
\eqref{eq:Bzudilin}, or, 
alternatively (cf.\ 
\cite[Lemma~4]{zud}, respectively \eqref{eq:rajout3} below) as
\begin{equation}\label{eq:definitionBNgras}
\mathbf{B}_{N_j}(m) = C_{N_j}^m \prod_{\ell=1}^{\varphi(N_j)} 
\frac{(r_{\ell,j}/N_j)_m}{m!},
\end{equation}
where $C_{N_j}$ and the $r_{\ell,j}$'s are defined as in
Subsection~\ref{ssec:zudconj}. 
Expression~\eqref{eq:definitionBNgras} will be 
useful in the first step below, 
while the direct use of 
Expression~\eqref{eq:Bzudilin} would lead to much more 
involved computations.

\subsection{First step}
Let us 
fix $j\in \{1, 2, \ldots, k\}$. We set 
$D_{N_j}:=N_j^{-\varphi(N_j)}C_{N_j}$, which is an integer.

We claim that
\begin{equation}\label{eq:step1} 
\frac{\mathbf{B}_{N_j}(v+up+np^{s+1})}{\mathbf{B}_{N_j}(up+np^{s+1})} = 
\frac{\mathbf{B}_{N_j}(v+up)}{\mathbf{B}_{N_j}(up)} + \mathcal{O}(p^{s+1}),
\end{equation}
where $\mathcal{O}(R)$ denotes an element of $R\mathbb Z_p$.
To prove~\eqref{eq:step1}, we observe 
that~(\footnote{Identities \eqref{eq:rajout9a} and \eqref{eq:rajout9} 
are immediate consequences of 
the alternative form \eqref{eq:definitionBNgras} of $\mathbf{B}_{N_j}$.
Zudilin 
used them in his proof of the following stronger version
of~\eqref{eq:step1}: 
$$
\frac{\mathbf{B}_{N_j}(v+up+np^{s+1})}{\mathbf{B}_{N_j}(up+np^{s+1})} = 
\frac{\mathbf{B}_{N_j}(v+up)}{\mathbf{B}_{N_j}(up)}\,\big(1 +
\mathcal{O}(p^{s+1})\big). 
$$
However, for this, he assumes that $p$ divides $N_j$ (see~\cite[Eq.~(35)]{zud}).
Here, we do not assume that $p$ divides $N_j$, and therefore we obtain
the weaker equality~\eqref{eq:step1}, which is fortunately enough for
our purposes.}) 
\begin{align} \label{eq:rajout9a}
\frac{\mathbf{B}_{N_j}(v+up+np^{s+1})}{\mathbf{B}_{N_j}(up+np^{s+1})} 
&=  
\frac{D_{N_j}^v \prod_{\ell=1}^{\varphi(N_j)}\prod_{i=1}^v 
\big(r_{\ell,j}+(i-1)N_j+uN_jp + nN_j p^{s+1}\big)}
{\big((v+up+np^{s+1})(v-1+up+np^{s+1})\cdots (1+up+np^{s+1})\big)^{\varphi(N_j)}}\\
&=  
\frac{\Big(D_{N_j}^v \prod_{\ell=1}^{\varphi(N_j)}\prod_{i=1}^v 
\big(r_{\ell,j}+(i-1)N_j+uN_jp \big)\Big)+\mathcal O(p^{s+1})}
{\big((v+up)(v-1+up)\cdots (1+up)\big)^{\varphi(N_j)}+\mathcal
O(p^{s+1})}.
\notag
\end{align}

If $v=0$, then \eqref{eq:step1} holds trivially.  
If $v>0$, then, together with the hypothesis $ v<p$, we infer that
$(v+up)(v-1+up)\cdots (1+up)$  
is not divisible by $p$, and thus we have 
\begin{multline*}
 \frac{1}
{\big((v+up)(v-1+up)\cdots (1+up)\big)^{\varphi(N_j)}+\mathcal{O}(p^{s+1})} \\
=
\frac{1}{\big((v+up)(v-1+up)\cdots (1+up)\big)^{\varphi(N_j)}} \, \big(1+
\mathcal{O}(p^{s+1}) \big). 
\end{multline*}
Hence, 
\begin{multline*}
\frac{\Big(D_{N_j}^v \prod_{\ell=1}^{\varphi(N_j)}\prod_{i=1}^v 
\big(r_{\ell,j}+(i-1)N_j+uN_jp \big)\Big)+\mathcal O(p^{s+1})}
{\big((v+up)(v-1+up)\cdots (1+up)\big)^{\varphi(N_j)}+\mathcal O(p^{s+1})}
\\
= \frac{D_{N_j}^v \prod_{\ell=1}^{\varphi(N_j)}\prod_{i=1}^v 
\big(r_{\ell,j}+(i-1)N_j+uN_jp \big)}
{\big((v+up)(v-1+up)\cdots (1+up)\big)^{\varphi(N_j)}}\\
+\frac{\mathcal{O}(p^{s+1})}
{\big((v+up)(v-1+up)\cdots (1+up)\big)^{\varphi(N_j)}},
\end{multline*}
which proves~\eqref{eq:step1} because 
\begin{equation} \label{eq:vup}
\frac{1}{(v+up)(v-1+up)\cdots (1+up)}\in \mathbb{Z}_p
\end{equation}
and  
\begin{equation}
\label{eq:rajout9}
\frac{D_{N_j}^v \prod_{\ell=1}^{\varphi(N_j)}\prod_{i=1}^v 
\big(r_{\ell,j}+(i-1)N_j+uN_jp\big)}
{\big((v+up)(v-1+up)\cdots (1+up)\big)^{\varphi(N_j)}}
=\frac{\mathbf{B}_{N_j}(v+up)}{\mathbf{B}_{N_j}(up)}   .
\end{equation}

A side result of~\eqref{eq:rajout9} (which was actually 
used to prove~\eqref{eq:step1}) is that  
$$
\frac{\mathbf{B}_{N_j}(v+up)}{\mathbf{B}_{N_j}(up)} \in \mathbb{Z}_p.
$$
We deduce from this fact and from~\eqref{eq:step1} that 
$$
\prod_{j=1}^k
\frac{\mathbf{B}_{N_j}(v+up+np^{s+1})}{\mathbf{B}_{N_j}(up+np^{s+1})}
=
\prod_{j=1}^k
\left(\frac{\mathbf{B}_{N_j}(v+up)}{\mathbf{B}_{N_j}(up)} +
\mathcal{O}(p^{s+1})\right)  
= \prod_{j=1}^k \frac{\mathbf{B}_{N_j}(v+up)}{\mathbf{B}_{N_j}(up)} +
\mathcal{O}(p^{s+1}) ,
$$
or, in other words, 
\begin{equation}\label{eq:main1}
\frac{\mathbf{B}_{\mathbf{N}}(v+up+np^{s+1})}
{\mathbf{B}_{\mathbf{N}}(up+np^{s+1})}
=  
\frac{\mathbf{B}_{\mathbf{N}}(v+up)}{\mathbf{B}_{\mathbf{N}}(up)} +
\mathcal{O}(p^{s+1}). 
\end{equation}

\subsection{Second step} Let us fix $j\in \{1, 2, \ldots, k\}$. 
The properties of $\Gamma_p$ imply that 
\begin{align}
\frac{\mathbf{B}_{N_j}(up+np^{s+1})}{\mathbf{B}_{N_j}(u+np^{s})} 
&= (-1)^{\mu_j-\eta_j}\frac{\prod_{i=1}^{\mu_j}\Gamma_p\big(1+\al_{i,j}(up+np^{s+1})\big)}
{\prod_{i=1}^{\eta_j}\Gamma_p\big(1+\be_{i,j}(up+np^{s+1})\big)} 
\label{eq:firstequality}\\
&= (-1)^{\mu_j-\eta_j}\frac{\prod_{i=1}^{\mu_j}\Gamma_p(1+\al_{i,j}up)+\mathcal{O}(p^{s+1})}
{\prod_{i=1}^{\eta_j}\Gamma_p(1+\be_{i,j}up)+\mathcal{O}(p^{s+1})} 
\label{eq:rajoutcorrectionbis} \\
&= (-1)^{\mu_j-\eta_j}\frac{\prod_{i=1}^{\mu_j}\Gamma_p(1+\al_{i,j}up)}
{\prod_{i=1}^{\eta_j}\Gamma_p(1+\be_{i,j}up)}\,\big(1
+\mathcal{O}(p^{s+1})\big) 
 \label{eq:rajoutcorrection}\\
&= \frac{\mathbf{B}_{N_j}(up)}{\mathbf{B}_{N_j}(u)} \,\big(1
+\mathcal{O}(p^{s+1})\big) ,
\label{eq:2}
\end{align}
where $(i)$ of Lemma~\ref{lem:gammap} is used 
to see~\eqref{eq:firstequality} and~\eqref{eq:2}, and 
$(ii)$ is used for~\eqref{eq:rajoutcorrectionbis}.
Equation~\eqref{eq:rajoutcorrection} holds because $\Gamma_p(m)$
is never divisible by $p$ for any integer $m$. 

Hence, taking the product over $j=1, 2, \ldots, k$, we obtain 
\begin{equation}\label{eq:main2}
\frac{\mathbf{B}_{\mathbf{N}}(up+np^{s+1})}{\mathbf{B}_{\mathbf{N}}(u+np^{s})}
=  
\frac{\mathbf{B}_{\mathbf{N}}(up)}{\mathbf{B}_{\mathbf{N}}(u)}\,\big(1
+ \mathcal{O}(p^{s+1})\big). 
\end{equation}

\subsection{Third step} 
We now multiply the right-hand and left-hand sides
of~\eqref{eq:main1} and~\eqref{eq:main2}. After simplification, we get
\begin{equation*}
\frac{\mathbf B_{\mathbf N}(v+up+np^{s+1})}{\mathbf B_{\mathbf N}(u+np^s)} = 
\frac{\mathbf B_{\mathbf N}(v+up)}{\mathbf B_{\mathbf N}(u)}\,\big(1+\mathcal{O}(p^{s+1})\big) 
+ \frac{\mathbf B_{\mathbf N}(up)}{\mathbf B_{\mathbf N}(u)} \,\mathcal{O}(p^{s+1}).
\end{equation*}
We can rewrite this as 
\begin{align}
\frac{\mathbf B_{\mathbf N}(v+up+np^{s+1})}{\mathbf B_{\mathbf N}(v+up)}
&=\frac{\mathbf B_{\mathbf N}(u+np^s)}{\mathbf B_{\mathbf
N}(u)}\,\big(1+\mathcal{O}(p^{s+1})\big) 
+\frac{\mathbf B_{\mathbf N}(up)}{\mathbf B_{\mathbf N}(u)}\cdot
\frac{\mathbf B_{\mathbf N}(u+np^s)}
{\mathbf B_{\mathbf N}(v+up)} \,\mathcal{O}(p^{s+1})
\notag
\\
&=\frac{\mathbf B_{\mathbf N}(u+np^s)}
{\mathbf B_{\mathbf N}(u)}
+\frac{\mathbf B_{\mathbf N}(u+np^s)}
{\mathbf B_{\mathbf N}(u)}\,\mathcal{O}(p^{s+1})
+\frac{\mathbf B_{\mathbf N}(u+np^s)}
{\mathbf B_{\mathbf N}(v+up)} \,\mathcal{O}(p^{s+1}),
\label{eq:quasifinale}
\end{align}
where the last line holds because 
$v_p\big(\mathbf B_{\mathbf N}(up)/\mathbf B_{\mathbf N}(u)\big)=0$.

If we compare \eqref{eq:DworkA} (with $A(m)=g(m)=\mathbf 
B_{\mathbf N}(m)$) 
and \eqref{eq:quasifinale}, we see that it only
remains to prove that we have  
\begin{equation}\label{eq:check}
\frac{\mathbf{B}_{\mathbf{N}}(u+np^{s})}{\mathbf{B}_{\mathbf{N}}(u)} \in 
\frac{\mathbf{B}_{\mathbf{N}}(n)}{\mathbf{B}_{\mathbf{N}}(v+up)} \,
\mathbb{Z}_p 
\quad
\textup{and}
\quad
\frac{\mathbf{B}_{\mathbf{N}}(u+np^{s})}{\mathbf{B}_{\mathbf{N}}(v+up)}\in 
\frac{\mathbf{B}_{\mathbf{N}}(n)}{\mathbf{B}_{\mathbf{N}}(v+up)} \,
\mathbb{Z}_p. 
\end{equation}
The first assertion in~\eqref{eq:check} can be rewritten as
\begin{equation} \label{eq:ass1} 
\frac{\mathbf{B}_{\mathbf{N}}(u+np^{s})}{\mathbf{B}_{\mathbf{N}}(n)}\cdot
\frac{\mathbf{B}_{\mathbf{N}}(v+up)}{\mathbf{B}_{\mathbf{N}}(u)} \in
\mathbb{Z}_p ,
\end{equation}
while the second assertion can be rewritten as 
\begin{equation} \label{eq:ass2} 
\frac{\mathbf{B}_{\mathbf{N}}(u+np^{s})}{\mathbf{B}_{\mathbf{N}}(n)}\in
\mathbb{Z}_p.
\end{equation}
Now, the assertion \eqref{eq:ass2} is the special case 
$w=u$, $m=n$ and $r=s$ of 
Lemma~\ref{lem:ultime}, while \eqref{eq:ass1} follows from 
\eqref{eq:ass2} combined with the special case 
$w=v$, $m=u$ and $r=1$ of Lemma~\ref{lem:ultime}.

\medskip
This completes the proof of the lemma.

\section{Proof of Lemma~\ref{lem:strat4}} \label{sec:6a}

The claim is trivially true if $p$ divides $m$.
We may therefore assume that $p$ does not divide $m$ for the rest of the proof.
Let us write $m=a+pj$, with $0< a<p$. Then
comparison with \eqref{eq:congrconj1a} shows that we are in a
very similar situation here. Indeed, we may derive \eqref{eq:110a}
from Lemma~\ref{lem:12a}. In order to see this, we observe that
\begin{align*}
H_{Lmp^s}-H_{L\fl{m/p}p^{s+1}}&=
\sum _{\ep=1} ^{Lap^s}\frac {1} {Ljp^{s+1}+\ep}\\
&=
\sum _{\ep=1} ^{\fl{La/p}}\frac {1} {Ljp^{s+1}+\ep p^{s+1}}
+
\underset{p^{s+1}\nmid \ep}{\sum _{\ep=1} ^{Lap^s}}\frac {1}
{Ljp^{s+1}+\ep}\\
&=\frac {1} {p^{s+1}}(H_{Lj+\fl{La/p}}-H_{Lj})+
\underset{p^{s+1}\nmid \ep}{\sum _{\ep=1} ^{Lap^s}}\frac {1}
{Ljp^{s+1}+\ep}.
\end{align*}
Because of $v_p(x+y)\ge\min\{v_p(x),v_p(y)\}$, this implies
\begin{equation} \label{eq:vH} 
v_p(H_{Lmp^s}-H_{L\fl{m/p}p^{s+1}})\ge
\min\{ -1-s+v_p(H_{Lj+\fl{La/p}}-H_{Lj}),-s\}.
\end{equation}
It follows that
\begin{multline} \label{eq:vBH}
v_p\Big(\mathbf B_{\mathbf N}(m)\big(H_{Lmp^s}-H_{L\fl{m/p}p^{s+1}}\big)\Big)\\
\ge
-1-s+\min\left\{v_p\Big(\mathbf 
B_{\mathbf N}(a+pj)(H_{Lj+\fl{La/p}}-H_{Lj})\big),
1+v_p\big(\mathbf B_{\mathbf N}(a+pj)\Big)\right\}.
\end{multline}
Use of Lemma~\ref{lem:12a} then completes the proof.

\section{Proof of Lemma~\ref{lem:12}} \label{sec:4}

We follow the same approach as the one of the proof of Lemma~\ref{lem:12a} in
Section~\ref{sec:9}. In particular, the first part below is completely
parallel to the proof of Lemma~\ref{lem:12a}. We nevertheless include it here
for the sake of readability and for later
reference. On the other hand, since Lemma~\ref{lem:12} makes a
stronger divisibility assertion 
than Lemma~\ref{lem:12a}, much more
work is needed to arrive there: the corresponding arguments form the
contents of the second and third part of this proof.

We start again by observing that
the assertion \eqref{eq:congrconj1} is trivially true if $\lfloor La/p\rfloor=0$, that is,
if $0\le a<p/L$. We may hence assume that $p/L\le a<p$ from now on.
A further assumption upon which we agree without loss of generality
for the rest of the proof is that $N_k=\max(N_1, \ldots, N_k)$.

\subsection{First part: a weak version of Lemma~\ref{lem:12}}
In a first step, we prove that
\begin{equation} \label{eq:congrconj2}
B_{\mathbf N}(a+pj)\left(H_{Lj+\lfloor La/p\rfloor} - H_{Lj}\right) \in
p\mathbb{Z}_p\ .
\end{equation}
(The reader should note the absence of the term
$M_{\mathbf{N}}/\Th_L$ in comparison with \eqref{eq:congrconj1}.)

\medskip
For the proof of \eqref{eq:congrconj2}, we note that
the $p$-adic valuation of $B_{\mathbf N}(a+pj)$ is equal to
$$v_p\big(B_{\mathbf N}(a+pj)\big)=
\sum _{i=1} ^{k}\sum _{\ell =1} ^{\infty}\left(\fl{\frac {N_i(a+pj)}
{p^\ell }}- N_i\fl{\frac {a+pj} {p^\ell }}\right).$$
Obviously, all the summands in this sum are non-negative, whence, in
particular,
\begin{equation} \label{eq:111}
v_p\big(B_{\mathbf N}(a+pj)\big)\ge
\sum _{\ell =1} ^{\infty}\left(\fl{\frac {N_k(a+pj)} {p^\ell }}-
N_k\fl{\frac {a+pj} {p^\ell }}\right).
\end{equation}
On the other hand, by definition of the harmonic numbers, we have
$$
H_{Lj+\lfloor La/p\rfloor} - H_{Lj}=\frac {1} {Lj+1}+\frac {1} {Lj+2}+\dots+
\frac {1} {Lj+\lfloor La/p\rfloor}.
$$
It therefore suffices to show that
\begin{equation} \label{eq:100} 
v_p\big(B_{\mathbf N}(a+pj)\big)\ge 1+ 
\max_{1\le\ep\le \lfloor La/p\rfloor}v_p(Lj+\ep).
\end{equation}
The lower bound on the right-hand side of \eqref{eq:111}
can, in fact, be simplified since $0\le a<p$; namely, we have
\begin{equation} \label{eq:104} 
\fl{\frac {a+pj} {p^\ell }}=\fl{\frac {j} {p^{\ell -1}}}.
\end{equation}

For a given integer 
$\ep$ with $1\le\ep\le \lfloor La/p\rfloor$, let
$Lj+\ep=p^d\be$, where $d=v_p(Lj+\ep)$. If we use this notation in 
\eqref{eq:111}, together with \eqref{eq:104}, we obtain 
\begin{equation} \label{eq:103} 
v_p\big(B_{\mathbf N}(a+pj)\big)\ge
\sum _{\ell =1} ^{\infty}\left(\fl{\frac {N_ka} {p^\ell }-\frac
{N_k\ep}
{Lp^{\ell -1}}+ \frac {N_k\be} {L} p^{d+1-\ell }}-
N_k\fl{-\frac {\ep}
{Lp^{\ell -1}}+ \frac {\be} {L} p^{d+1-\ell }}\right).
\end{equation}
Since $\ep\le \lfloor La/p\rfloor$, 
we have $\frac {N_ka} {p^\ell }-\frac {N_k\ep}
{Lp^{\ell -1}}\ge0$, whence
\begin{equation} \label{eq:101} 
\fl{\frac {N_ka} {p^\ell }-\frac {N_k\ep}
{Lp^{\ell -1}}+ \frac {N_k\be} {L} p^{d+1-\ell }}\ge
\fl{ \frac {N_k\be} {L} p^{d+1-\ell }}.
\end{equation}
Clearly, we also have
\begin{equation} \label{eq:102}
 \fl{-\frac {\ep}
{Lp^{\ell -1}}+ \frac {\be} {L} p^{d+1-\ell }}\le
\fl{ \frac {\be} {L} p^{d+1-\ell }}.
\end{equation}
If we use \eqref{eq:101} and \eqref{eq:102} 
in \eqref{eq:103}, then we obtain
\begin{equation} \label{eq:105} 
v_p\big(B_{\mathbf N}(a+pj)\big)\ge
\sum _{\ell =1} ^{\infty}\left(\fl{
 \frac {N_k\be} {L} p^{d+1-\ell }}-
N_k\fl{ \frac {\be} {L} p^{d+1-\ell }}\right).
\end{equation}
By the same argument as the one that 
we used in the second step of the proof of Lemma~\ref{lem:12a}
in Section~\ref{sec:9}, the rational
number $\frac{\beta p^{d+1-\ell}}{L}$ is not an integer.
However, the fact that $\be p^{d+1-\ell }/L$ is not an integer entails that
$$
 \frac {\be} {L} p^{d+1-\ell }-
\fl{ \frac {\be} {L} p^{d+1-\ell }}\ge\frac {1} {L},$$
as long as $\ell \le d+1$. Multiplication of both sides of this inequality by
$N_k$ leads to the chain of inequalities
$$
 \frac {N_k\be} {L} p^{d+1-\ell }-
N_k\fl{ \frac {\be} {L} p^{d+1-\ell }}\ge\frac {N_k} L\ge
1
$$
(it is here where we use the assumption $L\le
N_k=\max(N_1,\dots,N_k)$),
whence
$$\fl{\frac {N_k\be} {L} p^{d+1-\ell }}-N_k\fl{ \frac {\be} {L}
p^{d+1-\ell }}\ge 1, 
$$
provided $\ell \le d+1$.
Use of this estimation in \eqref{eq:105} gives
$$v_p\big(B_{\mathbf N}(a+pj)\big)\ge d+1=1+v_p(Lj+\ep).$$
This completes the proof of \eqref{eq:100}, and, hence, of
\eqref{eq:congrconj2}.

\medskip
For later use, we record that we have in particular shown that
for any 
$$D\le 1+\max_{1\le\ep\le \lfloor La/p\rfloor}v_p(Lj+\ep)$$ 
we have
\begin{equation} \label{eq:sumest} 
\sum _{\ell =2} ^{D}\left(\fl{\frac {N_k(a+pj)} {p^\ell }}-
N_k\fl{\frac {a+pj} {p^\ell }}\right)\ge D-1.
\end{equation}

\medskip
We now embark on the proof of \eqref{eq:congrconj1} itself.

\subsection{Second part: the case $j=0$}
In this case, we want to prove that
\begin{equation} \label{eq:congrconj4} 
B_{\mathbf N}(a)H_{\lfloor La/p\rfloor} \in p\frac {M_{\mathbf{N}}}
{\Th_L}\mathbb{Z}_p\ ,
\end{equation}
or, using \eqref{eq:J} (in the other direction), equivalently
\begin{equation} \label{eq:congrconj5} 
B_{\mathbf N}(a)H_{La} \in \frac {M_{\mathbf{N}}}
{\Th_L}\mathbb{Z}_p\ ,
\end{equation}
The reader should keep in mind that we still assume
that $p/L\le a<p$, so that, in particular, $a>0$. 

If $p>N_k=\max(N_1,\dots,N_k)$, then our claim, in the form
\eqref{eq:congrconj4}, reduces to\break 
$B_{\mathbf N}(a)H_{\lfloor La/p\rfloor} \in
p\mathbb{Z}_p$, 
which is indeed true because of \eqref{eq:congrconj2} with $j=0$.

Now let $p\le N_k$. By Lemma~\ref{lem:multinomial/N!} and the definition of
$\Th_L$, our claim, this time in the form
\eqref{eq:congrconj5}, holds for $a=1$. So, let $a\ge2$ from now on.

In a similar way as we did for the expression in
\eqref{eq:congrconj2}, we bound the $p$-adic valuation of the
expression in \eqref{eq:congrconj5} from below:
\begin{align} \notag
v_p\big(B_{\mathbf N}(a)H_{La}\big)
&=
\sum _{i=1} ^{k}\sum _{\ell =1} ^{\infty}\left(\fl{\frac {N_ia} {p^\ell }}-
N_i\fl{\frac {a} {p^\ell }}\right)+
v_p(H_{La})\notag\\
&\ge\sum _{i=1} ^{k}\sum _{\ell =1} ^{\infty}\fl{\frac {N_ia} {p^\ell }}
-\fl{\log_p La}\notag\\
&\ge\sum _{i=1} ^{k}\sum _{\ell =1} ^{\infty}\fl{\frac {2N_i} {p^\ell }}
-\fl{\log_p Lp}\notag\\
&\ge\fl{\frac {2N_k} {p}}+
\sum _{\ell =2} ^{\infty}\fl{\frac {2N_k} {p^\ell }}+
\sum _{i=1} ^{k-1}\sum _{\ell =1} ^{\infty}\fl{\frac {2N_i} {p^\ell }}
-\fl{\log_p L}-1
\notag\\
&\ge\fl{\frac {N_k} {p}}+
\sum _{i=1} ^{k}\sum _{\ell =1} ^{\infty}\fl{\frac {N_i} {p^\ell }}
-\fl{\log_p L}-1
\label{eq:unglA0}\\
&\ge\max\left\{1,\fl{L/p}\right\}+
\sum _{i=1} ^{k}v_p(N_i!)
-\fl{\log_p L}-1.
\label{eq:unglA}
\end{align}
If $p=2$, then we can continue the estimation \eqref{eq:unglA} as
\begin{equation} \label{eq:unglB} 
v_2\big(B_{\mathbf N}(a)H_{La}\big)\ge 
\sum _{i=1} ^{k}v_2(N_i!)
-\fl{\log_2 L}=
v_2\big(M_{\mathbf{N}}/\Th_L\big),
\end{equation}
where we used the simple fact $v_2(H_L)=-\fl{\log_2L}$ to obtain the equality.
(In fact, at this point it was not necessary to consider the case
$p=2$ because $a<p$ and because we assumed $a\ge2$. However, we shall
re-use the present estimations later in the third part of the current
proof, in a context where $a=1$ is allowed.)

From now on let $p\ge3$.
We use the fact that
\begin{equation} \label{eq:log}
x\ge\fl{\log_px}+2 
\end{equation}
for all integers $x\ge2$ and primes $p\ge3$.
Thus, in the case that $L\ge2p$, 
the estimation \eqref{eq:unglA} can be continued as
$$
v_p\big(B_{\mathbf N}(a)H_{La}\big)
\ge 
1+\fl{\log_p\fl{L/p}}+
\sum _{i=1} ^{k}v_p(N_i!)
-\fl{\log_p L}
\ge 
\sum _{i=1} ^{k}v_p(N_i!)=v_p( M_{\mathbf{N}}),
$$
implying \eqref{eq:congrconj5} in this case.
If $p\le L<2p$, 
then the estimation \eqref{eq:unglA} can be continued as
\begin{align*}
v_p\big(B_{\mathbf N}(a)H_{La}\big)
&\ge 
1+
\sum _{i=1} ^{k}v_p(N_i!)
-2=v_p(M_{\mathbf{N}}/\Th_L),
\end{align*}
implying \eqref{eq:congrconj5} in this case also.
Finally, if $L<p$, it follows from \eqref{eq:unglA} that 
\begin{align*}
v_p\big(B_{\mathbf N}(a)H_{La}\big)
&\ge 
1+
\sum _{i=1} ^{k}v_p(N_i!)
-1=v_p(M_{\mathbf{N}}),
\end{align*}
implying \eqref{eq:congrconj5} also in this final case.
Thus, \eqref{eq:congrconj4} is established.

\subsection{Third part: the case $j>0$}
Now let $j>0$.
If $p>N_k=\max(N_1,\dots,N_k)$, then \eqref{eq:congrconj1} reduces to 
\begin{equation*}
B_{\mathbf N}(a+pj)\left(H_{Lj+\lfloor La/p\rfloor} - H_{Lj}\right) \in
p\mathbb{Z}_p\ ,
\end{equation*}
which is 
again true because of \eqref{eq:congrconj2}.

Now let $p\le N_k$. The reader should keep in mind that we still assume
that $p/L\le a<p$, so that, in particular, $a>0$. 
In a similar way as we did for the expression in
\eqref{eq:congrconj2}, we bound the $p$-adic valuation of the
expression in \eqref{eq:congrconj1} from below. For the sake of
convenience, we write $T_1$ for $\max_{1\le\ep\le \lfloor
La/p\rfloor}v_p(Lj+\ep)$ and $T_2$ for $\fl{\log_p (a+pj)}$. 
Since it is somewhat hidden where our assumption $j>0$ enters the
subsequent considerations, we point out to the reader that
$j>0$ implies that $T_2\ge1$; without this property the split of
the sum over $\ell$ into subsums in the 
chain of inequalities below would be impossible.
So, using the above notation, we have (the detailed explanations for the
various steps are given immediately after the 
following chain of estimations)
{\allowdisplaybreaks
\begin{align} \notag
v_p\Big(&B_{\mathbf N}(a+pj)\left(H_{Lj+\lfloor La/p\rfloor} - H_{Lj}\right)\Big)\\
\notag
&=
\sum _{i=1} ^{k}\sum _{\ell =1} ^{\infty}\left(\fl{\frac {N_i(a+pj)}
{p^\ell }}- 
N_i\fl{\frac {a+pj} {p^\ell }}\right)+
v_p\big(H_{Lj+\lfloor La/p\rfloor} - H_{Lj}\big)\\
&=
\fl{\frac {N_k(a+pj)} {p }}-
N_k\fl{\frac {a+pj} {p }}
+
\sum _{\ell =2} ^{\min\{1+T_1,T_2\}}
\left(\fl{\frac {N_k(a+pj)} {p^\ell }}-
N_k\fl{\frac {a+pj} {p^\ell }}\right)
\notag\\
\notag
&\kern1cm
+
\sum _{\ell =\min\{1+T_1,T_2\}+1} ^{\infty}
\left(\fl{\frac {N_k(a+pj)} {p^\ell }}-
N_k\fl{\frac {a+pj} {p^\ell }}\right)\\
\notag
&\kern1cm
+
\sum _{i=1} ^{k-1}\sum _{\ell =1} ^{\infty}\left(\fl{\frac {N_i(a+pj)}
{p^\ell }}- 
N_i\fl{\frac {a+pj} {p^\ell }}\right)
+v_p\big(H_{Lj+\lfloor La/p\rfloor} - H_{Lj}\big)
\notag
\\
\notag
&\ge
\fl{\frac {N_ka} {p }}+\min\{1+T_1,T_2\}-1
+
\sum _{i=1} ^{k}\sum _{\ell =T_2+1} ^{\infty}
\left(\fl{\frac {N_i(a+pj)} {p^\ell }}-
N_i\fl{\frac {a+pj} {p^\ell }}\right)\\
&\kern1cm
+v_p\big(H_{Lj+\lfloor La/p\rfloor} - H_{Lj}\big)
\label{eq:ungl1}
\\
&\ge
\fl{\frac {N_ka} {p }}
+T_1+v_p\big(H_{Lj+\lfloor La/p\rfloor} - H_{Lj}\big)
+\min\{0,T_2-T_1-1\}
\notag\\
&\kern1cm
+
\sum _{i=1} ^{k}\sum _{\ell =\fl{\log_p(a+pj)}+1}
^{\infty}\left(\fl{\frac {N_i(a+pj)} {p^\ell }}- 
N_i\fl{\frac {a+pj} {p^\ell }}\right)
\label{eq:ungl2}\\
&\ge 
\max\left\{1,\fl{L/p}\right\}
+\min\{0,T_2-T_1-1\}
+
\sum _{i=1} ^{k}\sum _{\ell =1} ^{\infty}\fl{\frac {N_i} {p^\ell
}\cdot\frac {a+pj} {p^{\fl{\log_p(a+pj)}}}}
\label{eq:ungl3}\\
&\ge 
\max\left\{1,\fl{L/p}\right\}+\fl{\log_p (a+pj)}
-\fl{\log_p\big(Lj+\lfloor La/p\rfloor\big)}-1
\notag\\
&\kern1cm
+
\sum _{i=1} ^{k}\sum _{\ell =1} ^{\infty}\fl{\frac {N_i} {p^\ell
}\cdot\frac {a+pj} {p^{\fl{\log_p(a+pj)}}}}
\label{eq:ungl4}
\\
&\ge 
\max\left\{1,\fl{L/p}\right\}+\fl{\log_p j}
-\fl{\log_p\big(Lj+\lfloor La/p\rfloor\big)}
+
\sum _{i=1} ^{k}\sum _{\ell =1} ^{\infty}\fl{\frac {N_i} {p^\ell }}
\label{eq:ungl5}
\\
&\ge 
\max\left\{1,\fl{L/p}\right\}+\fl{\log_p j}
-\fl{\log_p L}-\fl{\log_p\left(j+\frac {1} {L}\lfloor La/p\rfloor\right)}-1
\notag\\
&\kern1cm
+
\sum _{i=1} ^{k}v_p(N_i!)
\label{eq:ungl6}\\
&\ge 
\max\left\{1,\fl{L/p}\right\}
-\fl{\log_p L}-1+v_p(M_{\mathbf{N}}).
\label{eq:ungl7}
\end{align}
}%
Here, we used \eqref{eq:sumest} in order to get \eqref{eq:ungl1}.
To get \eqref{eq:ungl3}, we used the inequalities
\begin{equation} \label{eq:ungl6a} 
\fl{\frac {N_ka} {p}}\ge\fl{\frac {N_k} {p}}\ge
\max\left\{1,\fl{L/p}\right\}
\end{equation}
and 
\begin{equation} \label{eq:ungl100} 
T_1+v_p\big(H_{Lj+\lfloor La/p\rfloor} - H_{Lj}\big)\ge0.
\end{equation}
To get \eqref{eq:ungl4}, we used that
$$T_2-T_1-1\ge \fl{\log_p (a+pj)}
-\fl{\log_p\big(Lj+\lfloor La/p\rfloor\big)}-1$$
and
$$
\fl{\log_p (a+pj)}
-\fl{\log_p\big(Lj+\lfloor La/p\rfloor\big)}-1=
\fl{\log_p j}-\fl{\log_p\big(Lj+\lfloor La/p\rfloor\big)}
\le 0,
$$
so that
\begin{equation} \label{eq:ungl6b}
\min\{0,T_2-T_1-1\}\ge 
\fl{\log_p (a+pj)}
-\fl{\log_p\big(Lj+\lfloor La/p\rfloor\big)}-1.
\end{equation}
Next, to get \eqref{eq:ungl5}, we used
\begin{equation} \label{eq:ungl6d}
\fl{\frac {N_i} {p^\ell }\cdot\frac {a+pj}
{p^{\fl{\log_p(a+pj)}}}}\ge
\fl{\frac {N_i} {p^\ell }}.
\end{equation}
To get \eqref{eq:ungl6}, we used
\begin{equation} \label{eq:ungl6c}
\fl{\log_p\big(Lj+\lfloor La/p\rfloor\big)}
\le\fl{\log_p L}+\fl{\log_p\left(j+\frac {1} {L}\lfloor La/p\rfloor\right)}+1 .
\end{equation}
Finally, we used $\frac {1} {L}\lfloor
La/p\rfloor<1$ in order to get
\eqref{eq:ungl7}. 

If we now repeat the arguments after \eqref{eq:unglA},
then we see that the estimation \eqref{eq:ungl7} implies
\begin{equation} \label{eq:ungl8} 
v_p\Big(B_{\mathbf N}(a+pj)\left(H_{Lj+\lfloor La/p\rfloor} - H_{Lj}\right)\Big)
\ge
v_p\big(M_{\mathbf{N}}/\Th_L\big).
\end{equation}
This almost proves \eqref{eq:congrconj1}, our lower
bound on the $p$-adic valuation of the number in \eqref{eq:congrconj1}
is just by $1$ too low. 

In order to establish that \eqref{eq:congrconj1} is indeed correct, we
assume by contradiction that all the inequalities in the estimations
leading to \eqref{eq:ungl7} and finally to \eqref{eq:ungl8} 
are in fact equalities. 
In particular, the estimations in \eqref{eq:ungl6a} hold with equality only
if $a=1$ and, if $L$ should be at least $p$, also $\fl{N_k/p}=\fl{L/p}$.
We shall henceforth assume both of these two conditions.

If we examine the arguments after \eqref{eq:unglA} that led us from
\eqref{eq:ungl7} to \eqref{eq:ungl8}, then we see that they prove in
fact 
\begin{equation} \label{eq:ungl11} 
v_p\Big(B_{\mathbf N}(a+pj)\left(H_{Lj+\lfloor La/p\rfloor} - H_{Lj}\right)\Big)
\ge
1+v_p\big(M_{\mathbf{N}}/\Th_L\big)
\end{equation}
except if:

\bigskip
{\sc Case 1:} $p=2$ and $\fl{L/2}=1$;

{\sc Case 2:} $p\ge3$ and $p\le L<2p$;

{\sc Case 3:} $p=3$ and $\fl{L/3}=2$;

{\sc Case 4:} $L<p$.
\bigskip

In all other cases, there holds either strict inequality in
\eqref{eq:log} with $x=\fl{L/p}$, or one has $v_p(\Th_L)\ge1$
and is able to show 
$$
v_p\Big(B_{\mathbf N}(a+pj)\left(H_{Lj+\lfloor La/p\rfloor} - H_{Lj}\right)\Big)
\ge
v_p\big(M_{\mathbf{N}}\big),
$$ 
so that \eqref{eq:ungl11} is satisfied, as desired.
We now show that \eqref{eq:ungl11} holds as well in Cases~1--4, thus
completing the proof of \eqref{eq:congrconj1}.

\medskip
{\sc Case 1}. Let first $p=2$ and $L=2$. We then have
\begin{align*}
\min\{0,T_2-T_1-1\}&=\min\{0,\fl{\log_2(2j+1)}-v_2(2j+1)-1\}\\
&=
\min\{0,\fl{\log_2(2j+1)}-1\}=0>-1,
\end{align*}
in contradiction to having equality in \eqref{eq:ungl6b}.

On the other hand, 
if $p=2$ and $L=3$ then, because of equality in the second estimation
in \eqref{eq:ungl6a}, we must have $N_k=3$. We have
$$H_{Lj+\lfloor La/p\rfloor} - H_{Lj}
=H_{3j+1} - H_{3j}=\frac {1} {3j+1}.$$
If there holds equality in \eqref{eq:ungl6b}, then $Lj+\fl{La/p}=3j+1$ 
must be a power of $2$, say $3j+1=2^e$ or, equivalently, 
$j=(2^e-1)/3$. It follows that
$$
\fl{\frac {N_k} {p }\cdot\frac {a+pj}
{p^{\fl{\log_p(a+pj)}}}}
=\fl{\frac {3} {2 }\cdot\frac {1+2j}
{2^{\fl{\log_2(1+2j)}}}}
=\fl{\frac {3} {2 }\cdot\frac {2^{e+1}+1}
{3\cdot 2^{e-1}}}
=2>1=\fl{\frac {3} {2}}=\fl{\frac {N_k} {p}},
$$
in contradiction to having equality in \eqref{eq:ungl6d} with $\ell=1$.

\medskip
{\sc Case 2}. Our assumptions $p\ge3$ and $p\le L<2p$ imply 
$$
H_{Lj+\lfloor La/p\rfloor} - H_{Lj}
=H_{Lj+1} - H_{Lj}=\frac {1} {Lj+1}.
$$
Arguing as in the previous case, in order to have equality in 
\eqref{eq:ungl6b}, we must have $Lj+1=f\cdot p^e$ for some positive
integers $e$ and $f$ with $0<f<p$. Thus, $j=(f\cdot p^e-1)/L$ and
$p<L$. (If $p=L$ then $j$ would be non-integral.) It follows that
\begin{equation} \label{eq:ungl12}
\fl{\frac {N_k} {p }\cdot\frac {a+pj}
{p^{\fl{\log_p(a+pj)}}}}=
\fl{\frac {N_k} {p }\cdot\frac {f\cdot p^{e+1}+L-p}
{L\cdot p^{\fl{\log_p((f\cdot p^{e+1}+L-p)/L)}}}}. 
\end{equation}
If $f=1$, then we obtain from \eqref{eq:ungl12} that
$$
\fl{\frac {N_k} {p }\cdot\frac {a+pj}
{p^{\fl{\log_p(a+pj)}}}}=
\fl{\frac {N_k} {p }\cdot\frac {p^{e+1}+L-p}
{L\cdot p^{e-1}}}
\ge\fl{\frac {p^{e+1}+L-p} {p^{e}}}
>1=\fl{\frac {L} {p}}=\fl{\frac {N_k} {p}}, 
$$
in contradiction with having equality in \eqref{eq:ungl6d} with $\ell=1$.

On the other hand, if $f\ge2$, then we obtain from \eqref{eq:ungl12} that
$$
\fl{\frac {N_k} {p }\cdot\frac {a+pj}
{p^{\fl{\log_p(a+pj)}}}}\ge
\fl{\frac {f\cdot p^{e+1}+L-p}
{p^{e+1}}}\ge f>1=\fl{\frac {L} {p}}=\fl{\frac {N_k} {p}}, 
$$
again in contradiction with having equality in \eqref{eq:ungl6d} with $\ell=1$.

\medskip
{\sc Case 3}. Our assumptions $p=3$ and $\fl{L/3}=2$ imply 
$$
H_{Lj+\lfloor La/p\rfloor} - H_{Lj}
=H_{Lj+2} - H_{Lj}=\frac {1} {Lj+1}+\frac {1} {Lj+2}.
$$
Similar to the previous cases, in order to have equality in 
\eqref{eq:ungl6b}, we must have $Lj+\ep=f\cdot 3^e$ for some positive
integers $\ep,e,f$ with $0<\ep,f<3$. The arguments from Case~2 can now
be repeated almost verbatim. We leave the details to the reader.

\medskip
{\sc Case 4}. If $L<p$, then $p/L>1=a$, a contradiction to the
assumption that we made at the very beginning of this section.

\medskip
This completes the proof of the lemma.

\section{Proof of Lemma~\ref{lem:11}} \label{sec:6}

We proceed in the same way as in the proof of Lemma~\ref{lem:strat4}
in Section~\ref{sec:6a}.
Again, the claim is trivially true if $p$ divides $m$, so that
we may assume that $p$ does not divide $m$ for the rest of the proof.
Let us write $m=a+pj$, with $0< a<p$. Then
comparison with \eqref{eq:congrconj1} shows that we are in a
very similar situation here. Indeed, we may derive \eqref{eq:110}
from Lemma~\ref{lem:12}. In order to see this, we use \eqref{eq:vH}
to deduce
\begin{multline*}
v_p\Big(B_{\mathbf N}(m)\big(H_{Lmp^s}-H_{L\fl{m/p}p^{s+1}}\big)\Big)\\
\ge
-1-s+\min\left\{v_p\Big(B_{\mathbf N}(a+pj)(H_{Lj+\fl{La/p}}-H_{Lj})\big),
1+v_p\big(B_{\mathbf N}(a+pj)\Big)\right\}.
\end{multline*}
Use of Lemmas~\ref{lem:multinomial/N!} and \ref{lem:12} 
then completes the proof.

\section{The equivalence of Zudilin's and our definition of $\mathbf
H_N(m)$} \label{sec:H_N}

Zudilin's definition of the quantity $\mathbf H_N(m)$ deviates from
\eqref{eq:rajout2}. In this final section, we show that our
definition is equivalent to Zudilin's.

\begin{lem} \label{lem:rajout2} 
Let $m$ be a non-negative integer, and let $N$ be a positive integer
with associated parameters $\al_i, \beta_i, 
\mu, \eta$ {\em(}that is, given by
\eqref{eq:aj} and \eqref{eq:bj}, respectively{\em)}. Then
\begin{equation*}
\mathbf{H}_N(m)
=\sum_{j=1}^{\varphi(N)} H(r_j/N,m) - \varphi(N)H(1,m),
\end{equation*}
where $H(x,m):= \sum_{n=0}^{m-1}\frac{1}{x+n}$, and where
$r_j\in \{1, 2, \ldots, N\}$ form the residue 
classes modulo $N$ which are coprime to $N$. As before,
$\varphi(\,.\,)$ denotes Euler's totient function.
\end{lem}
\begin{proof} 
For $N=1$, we have $\mathbf{H}_1(m)=0$, so that the assertion of
the lemma holds trivially. Therefore, from now on, we assume $N\ge2$.

We claim that, for any real number $m\ge0$, we have
\begin{equation}\label{eq:rajout3}
\frac{C_N^m}{
\Gamma(m+1)^{\varphi(N)}}\prod_{j=1}^{\varphi(N)}\frac{\Gamma(m+r_j/N)}{\Gamma(r_j/N)}
=  
\frac{\prod_{j=1}^\mu \Gamma(\alpha_j m+1)}{\prod_{j=1}^\eta
\Gamma(\beta_j m+1)}, 
\end{equation}
where $\Ga(x)$ denotes the gamma function.
This generalises Zudilin's identity~\eqref{eq:Bzudilin}
to real values of $m$. We essentially extend his proof to
real $m$, using the well-known formula~\cite[p.~23, Theorem~1.5.2]{AAR}
\begin{equation} \label{eq:Ga}
\Ga(a)\,
\Ga\left(a+\frac {1} {n}\right)\,
\Ga\left(a+\frac {2} {n}\right)\cdots
\Ga\left(a+\frac {n-1} {n}\right)=
n^{\frac {1} {2}-an}(2\pi)^{(n-1)/2}\Ga(an),
\end{equation}
valid for real numbers $a$ and positive integers $n$ such that $aN$
is not an integer $\le0$. Indeed, as in the Introduction,
let $p_1, p_2, \dots, p_\ell$ denote the distinct prime factors   
of $N$. 
(It should be noted that there is at least one such prime factor due
to our assumption $N\ge2$.)
Furthermore, for a subset $J$ of $\{1,2,\dots,\ell\}$, let
$p_J$ denote the product $\prod _{j\in J} ^{}p_j$ of corresponding prime
factors of $N$. 
(In the case that $J=\emptyset$, the empty product must be interpreted
as $1$.) Then, by the principle of
inclusion-exclusion, we can rewrite the left-hand side of
\eqref{eq:rajout3} in the form
$$
\frac{C_N^m}{
\Gamma(m+1)^{\varphi(N)}}\cdot
\frac {
\underset{\vert J\vert\text { even}}
{\prod _{J\subseteq \{1,2,\dots,\ell\}} ^{}}
\prod _{i=1} ^{N/p_J}\Ga\left(m+\frac {ip_J} {N}\right)} 
{
\underset{\vert J\vert\text { odd}}
{\prod _{J\subseteq \{1,2,\dots,\ell\}} ^{}}
\prod _{i=1} ^{N/p_J}\Ga\left(m+\frac {ip_J} {N}\right)}
\cdot
\frac {
\underset{\vert J\vert\text { odd}}
{\prod _{J\subseteq \{1,2,\dots,\ell\}} ^{}}
\prod _{i=1} ^{N/p_J}\Ga\left(\frac {ip_J} {N}\right)} 
{
\underset{\vert J\vert\text { even}}
{\prod _{J\subseteq \{1,2,\dots,\ell\}} ^{}}
\prod _{i=1} ^{N/p_J}\Ga\left(\frac {ip_J} {N}\right)}  .
$$
To each of the products over $i$, formula~\eqref{eq:Ga} can be
applied. As a result, we obtain the expression
\begin{multline} \label{eq:Gaexpr}
\frac{C_N^m}{
\Gamma(m+1)^{\varphi(N)}}\cdot
\frac {
\underset{\vert J\vert\text { even}}
{\prod _{J\subseteq \{1,2,\dots,\ell\}} ^{}}
\left(\frac {N} {p_J}\right)^{-\left(m+\frac {p_J} {N}\right)\frac {N}
{p_J}}
\Ga\left(m\frac {N} {p_J}+1\right)} 
{
\underset{\vert J\vert\text { odd}}
{\prod _{J\subseteq \{1,2,\dots,\ell\}} ^{}}
\left(\frac {N} {p_J}\right)^{-\left(m+\frac {p_J} {N}\right)\frac {N}
{p_J}}
\Ga\left(m\frac {N} {p_J}+1\right)} 
\cdot
\frac {
\underset{\vert J\vert\text { odd}}
{\prod _{J\subseteq \{1,2,\dots,\ell\}} ^{}}
\Ga\left(1\right)} 
{
\underset{\vert J\vert\text { even}}
{\prod _{J\subseteq \{1,2,\dots,\ell\}} ^{}}
\Ga\left(1\right)} \\
=
\frac{C_N^m}{
\Gamma(m+1)^{\varphi(N)}}\cdot
\frac {
\underset{\vert J\vert\text { even}}
{\prod _{J\subseteq \{1,2,\dots,\ell\}} ^{}}
\left(\frac {N} {p_J}\right)^{-m {N}/{p_J}}
\Ga\left(m\frac {N} {p_J}+1\right)} 
{
\underset{\vert J\vert\text { odd}}
{\prod _{J\subseteq \{1,2,\dots,\ell\}} ^{}}
\left(\frac {N} {p_J}\right)^{-m {N}/{p_J}}
\Ga\left(m\frac {N} {p_J}+1\right)} ,
\end{multline}
where the simplification in the exponent of $N/p_J$ is due to the
fact that there are as many subsets of even cardinality of a given 
non-empty set as there are subsets of odd cardinality. 
Since, again by inclusion-exclusion, 
\begin{equation} \label{eq:phi}
\underset{\vert J\vert\text { even}}
{\sum _{J\subseteq \{1,2,\dots,\ell\}} ^{}}\frac {N} {p_J}
-\underset{\vert J\vert\text { odd}}
{\sum _{J\subseteq \{1,2,\dots,\ell\}} ^{}}\frac {N} {p_J}
=N\prod _{p\mid N} ^{}\left(1-\frac {1} {p}\right)
=\ph(N),
\end{equation}
we have
$$
\frac{1}{
\Gamma(m+1)^{\varphi(N)}}\cdot
\frac {
\underset{\vert J\vert\text { even}}
{\prod _{J\subseteq \{1,2,\dots,\ell\}} ^{}}
\Ga\left(m\frac {N} {p_J}+1\right)} 
{
\underset{\vert J\vert\text { odd}}
{\prod _{J\subseteq \{1,2,\dots,\ell\}} ^{}}
\Ga\left(m\frac {N} {p_J}+1\right)} =
\frac{\prod_{j=1}^\mu \Gamma(\alpha_j m+1)}{\prod_{j=1}^\eta
\Gamma(\beta_j m+1)}
$$
and
$$
\frac {
\underset{\vert J\vert\text { even}}
{\prod _{J\subseteq \{1,2,\dots,\ell\}} ^{}}
{N}^{-m {N}/{p_J}}}
{
\underset{\vert J\vert\text { odd}}
{\prod _{J\subseteq \{1,2,\dots,\ell\}} ^{}}
{N}^{-m {N}/{p_J}}}=
N^{-m\ph(N)}.
$$
Finally, consider a fixed prime number dividing $N$, 
$p_j$ say. Then, using again \eqref{eq:phi}, we
see that the exponent of $p_j$ in the expression
\eqref{eq:Gaexpr} is
$$
-\frac {m} {p_j}\underset{\vert J\vert\text { odd},\,j\notin J}
{\sum _{J\subseteq \{1,2,\dots,\ell\}} ^{}}\frac {N} {p_J}
+\frac {m} {p_j}\underset{\vert J\vert\text { even},\,j\notin J}
{\sum _{J\subseteq \{1,2,\dots,\ell\}} ^{}}\frac {N} {p_J}
=\frac {mN} {p_j}\underset{p\ne p_j}
{\prod _{p\mid N} ^{}}\left(1-\frac {1} {p}\right)
=\frac {m} {p_j}\frac {\ph(N)} {1-\frac {1} {p_j}}=\frac {m\varphi(N)}
{p_j-1}.
$$
If all these observations are used in \eqref{eq:Gaexpr}, we arrive at
the right-hand side of \eqref{eq:rajout3}.

Now, let us call $b(m)$ the function defined by both sides
of~\eqref{eq:rajout3}, and
let $\psi(x)=\Gamma'(x)/\Gamma(x)$ be the
digamma function.  
We will use the well-known property 
(see~\cite[p.~13, Theorem~1.2.7]{AAR}) that
$
\psi(x+n) - \psi(x) = H(x,n)
$
for real numbers $x>0$ and integers $n\ge 0$.

By taking the logarithmic derivative of the right-hand 
side of~\eqref{eq:rajout3}, we have  
\begin{align}
\frac{b'(m)}{b(m)} &= \sum_{j=1}^\mu \al_j \psi(\al_jm+1) -
\sum_{j=1}^\eta \be_j \psi(\be_jm+1) \notag 
\\
         &= \sum_{j=1}^\mu \al_j \big(\psi(1)+H_{\al_jm}\big) -
\sum_{j=1}^\eta \be_j\big(\psi(1)+ H_{\be_jm}\big) \notag 
\\
 &= \sum_{j=1}^\mu \al_j H_{\al_jm} - \sum_{j=1}^\eta \be_j
H_{\be_jm}\notag\\
&=\mathbf H_N(m), \label{eq:rajout:4} 
\end{align}
because $\sum_{j=1}^\mu \al_j = \sum_{j=1}^\eta \be_j$. It also
follows that $b'(0)/b(0)=0$. 

On the other hand,  by taking the logarithmic derivative of the 
left-hand side of~\eqref{eq:rajout3}, we also have  
\begin{equation*}
\frac{b'(m)}{b(m)} = \log(C_N) + \sum_{j=1}^{\varphi(N)} \psi(m+r_j/N)
- \varphi(N) \psi(m+1). 
\end{equation*}
Since $b'(0)/b(0)=0$, we have 
$
 \log(C_N) = -\sum_{j=1}^{\varphi(N)} \psi(r_j/N) + \varphi(N) \psi(1)
$
and therefore, 
\begin{align}
\frac{b'(m)}{b(m)} &= \sum_{j=1}^{\varphi(N)} \big(\psi(m+r_j/N)-
\psi(r_j/N)\big) - \varphi(N) \big(\psi(m+1)-\psi(1)\big)\notag 
\\
      &=  \sum_{j=1}^{\varphi(N)} H(r_j/N,m) - \varphi(N) H(1,m).
\label{eq:rajout:5}
\end{align}
The lemma follows by equating the expressions~\eqref{eq:rajout:4}
and~\eqref{eq:rajout:5} obtained for $b'(m)/b(m)$. 
\end{proof}

\section*{Acknowledgements}
The authors are extremely grateful to Alessio Corti 
and Catriona Maclean for illuminating discussions 
concerning the geometric side of our work. 
They are furthermore indebted to the anonymous referees for their
criticism and suggestions, which helped to improve the presentation
of the material considerably.

\def\refname{Bibliography}

\end{document}